\documentclass[12pt, a4paper]{amsart}
\usepackage{amsmath}
\usepackage{geometry,amsthm,graphics,tabularx,amssymb,shapepar}
\usepackage[cmyk]{xcolor}
\usepackage{appendix}
\usepackage{amscd}
\usepackage[all]{xypic}
\usepackage{ulem}
\usepackage{bm}
\usepackage{extarrows}

\makeatletter
\newcommand*{\rom}[1]{\expandafter\@slowromancap\romannumeral #1@}
\makeatother

\newcommand{\BC}{{\mathbb {C}}}

\newcommand{\BN}{{\mathbb {N}}}

\newcommand{\BR}{{\mathbb {R}}}

\newcommand{\CI}{{\mathcal {I}}}

\newcommand{\CO}{{\mathcal {O}}}

\newcommand{\RC}{{\mathrm {C}}}
\newcommand{\RD}{{\mathrm {D}}}

\newcommand{\Ri}{{\mathrm{i}}}

\newcommand{\RT}{{\mathrm {T}}}
\newcommand{\RU}{{\mathrm {U}}}

\newcommand{\Ad}{{\mathrm{Ad}}}

\newcommand{\Hom}{{\mathrm{Hom}}}

\newcommand{\Lie}{{\mathrm{Lie}}}

\newcommand{\Spec}{{\mathrm{Spec}}}

\newcommand{\wh}{\widehat}

\newcommand{\ad}{\operatorname{ad}}

\newcommand{\Ev}{\operatorname{Ev}}

\newcommand{\g}{\mathfrak g}

\newcommand{\h}{\mathfrak h}

\newcommand{\q}{\mathfrak q}

\renewcommand{\l}{\mathfrak l}

\newcommand{\m}{\mathfrak m}

\newcommand{\C}{\mathbb{C}}
\newcommand{\R}{\mathbb R}

\newcommand{\M}{\mathbf{M}}

\newcommand{\G}{\mathbf{G}}

\newcommand{\la}{\langle}
\newcommand{\ra}{\rangle}

\newcommand{\be}{\begin {equation}}
\newcommand{\ee}{\end {equation}}
\newcommand{\bee}{\begin {equation*}}
\newcommand{\eee}{\end {equation*}}

\newcommand{\qaq}{\quad\textrm{and}\quad}

\newcommand{\cf}{\textit{cf.}~}

\renewcommand{\mid}{\,:\,}

\theoremstyle{Theorem}

\theoremstyle{Theorem}

\theoremstyle{Theorem}

\theoremstyle{Theorem}

\theoremstyle{Plain}

\theoremstyle{remark}

\theoremstyle{remark}

\theoremstyle{Definition}
\newtheorem{dfn}{Definition}[section]

\newtheorem{cord}[dfn]{Corollary}
\newtheorem{prpd}[dfn]{Proposition}
\newtheorem{thmd}[dfn]{Theorem}
\newtheorem{lemd}[dfn]{Lemma}
\newtheorem{exampled}[dfn]{Example}
\newtheorem{remarkd}[dfn]{Remark}
\numberwithin{equation}{section}
\begin{document}

\title[Formal Lie groups]{Lie pairs and formal Lie groups}

\author[F. Chen]{Fulin Chen}
\address{School of Mathematical Sciences, Xiamen University,
	Xiamen, 361005, China} \email{chenf@xmu.edu.cn}

\author[B. Sun]{Binyong Sun}
\address{Institute for Advanced Study in Mathematics and New Cornerstone Science Laboratory, Zhejiang University,  Hangzhou, 310058, China}
\email{sunbinyong@zju.edu.cn}

\author[C. Wang]{Chuyun Wang}
\address{School of Mathematics and Statistics, Hainan University, Haikou, 570228, China}
\email{chuunw@amss.ac.cn}


\begin{abstract}
In a previous paper, we introduce and study formal manifolds, which generalize smooth manifolds. 
In this paper, we establish the basic theory of formal Lie groups, which are group objects in the category of formal manifolds.
In particular,  extending  the classical
formal Lie theory theorem, we prove that the category of formal Lie groups is equivalent to the category of Lie pairs.

\end{abstract}

\maketitle

\tableofcontents

\section{Introduction}
The formal Lie theory theorem establishes an equivalence between the
category of finite-dimensional complex Lie algebras and that of finite-dimensional complex formal group laws.
In this paper, we will establish the basic theory of formal Lie groups, and generalize the formal Lie theory theorem to the setting of Lie pairs and formal Lie groups (see Theorem \ref{thm:eqFP}).

\subsection{Lie pairs}
We begin by introducing the concept of Lie pairs.
\begin{dfn}\label{df:liepair}
	A \textbf{Lie pair} is a pair $(\mathfrak{q},L)$ consisting of a finite-dimensional complex Lie algebra $\mathfrak{q}$ and a real Lie group $L$, together with
	a   representation
	\[
	\Ad:\ L\curvearrowright \mathfrak{q}, \qquad (g,\eta )\mapsto \Ad_g \eta
	\]
	of $L$ on $\mathfrak{q}$ as Lie algebra automorphisms, as well as an injective Lie algebra homomorphism
	\[
	\iota:\ \mathfrak{l}\rightarrow \mathfrak{q} \quad (\text{$\mathfrak{l}$ is the complexified Lie algebra of $L$}),
	\]
	subject to the following conditions:
	\begin{itemize}
		\item
		$\iota$ is $L$-equivariant, where $\mathfrak{l}$ carries the adjoint representation of $L$, and $\mathfrak{q}$ carries the representation $\Ad: L\curvearrowright\mathfrak{q}$; and
		\item
		the differential $\ad: \mathfrak{l}\curvearrowright \mathfrak{q}$ of $\Ad: L\curvearrowright \mathfrak{q}$ equals the action
		\[
		\mathfrak{l}\curvearrowright \mathfrak{q},\qquad (\tau,\eta)\mapsto [\iota(\tau),\eta].
		\]
	\end{itemize}
\end{dfn}

Throughout this paper, all Lie groups are assumed to be Hausdorff (and thus they are paracompact). They may or may not have countably many connected components.
Under the notation and assumptions of Definition \ref{df:liepair}, we refer to the quadruple $(\mathfrak{q}, L, \Ad, \iota)$, or simply the pair $(\mathfrak{q}, L)$ when $\Ad$ and $\iota$ are understood, as a Lie pair.
The notion of Lie pairs unifies finite-dimensional complex Lie algebras and real Lie groups.

Note that every finite-dimensional complex Lie algebra $\mathfrak{q}$ determines a Lie pair $(\mathfrak{q},\{e\})$, while every real Lie group $L$ determines a Lie pair $(\mathfrak{l},L)$. Here and henceforth, the notation $\{e\}$ denotes the trivial group, and we use the corresponding lowercase Gothic letter to denote the complexified Lie algebra of a real Lie group.

All Lie pairs, together with their homomorphisms as defined below, form a category.

\begin{dfn}\label{df:morLiepair} Let $(\mathfrak{q}_1,L_1)$ and $(\mathfrak{q}_2,L_2)$ be two Lie pairs.
	A homomorphism
    \[(d\varphi,\underline \varphi): \ (\mathfrak{q}_1,L_1)\rightarrow (\mathfrak{q}_2,L_2) \] of Lie pairs is a Lie algebra homomorphism $d\varphi: \mathfrak{q}_1\rightarrow \mathfrak{q}_2$ as well as a Lie group homomorphism $\underline \varphi : L_1\rightarrow L_2$
       such that the diagrams
	\be \label{eq:comofhomofLiepairs}
	\begin{CD}
		L_1\times \mathfrak{q}_1 @> \mathrm{Ad} >>\mathfrak{q}_1 \\
		@V \underline\varphi\times d\varphi VV           @V  V d\varphi V\\
		L_2\times \mathfrak{q}_2@> \mathrm{Ad} >> \mathfrak{q}_2,
	\end{CD}
	\qquad\qquad\text{and}\qquad\qquad\qquad\quad 
	\begin{CD}
		\mathfrak{l}_1 @> \iota >>\mathfrak{q}_1 \\
		@V \text{differential of }\underline \varphi VV           @VV d\varphi V\\
		\mathfrak{l}_2 @> \iota   >> \mathfrak{q}_2
	\end{CD}
	\ee
	commute. 
\end{dfn}

In the literature, representations of Lie pairs $(\mathfrak{q},L)$ have primarily been studied under the assumption that $L$ is compact (see \cite{KV}). However, for applications to the theory of automorphic forms, it is desirable to develop a theory of smooth representations for general Lie pairs $(\mathfrak{q},L)$, where $L$ is not necessarily compact. In a complementary direction, we have laid the foundations of a theory of formal manifolds in \cite{CSW1, CSW2, CSW4}, generalizing the classical theory of smooth manifolds. Within this framework, a formal Lie group is defined as a group object in the category of formal manifolds. Since general Lie pairs correspond to formal Lie groups (see Theorem \ref{thm:eqFP}) and their smooth representations are often realized on various function spaces over formal manifolds, a thorough understanding of formal Lie groups is essential for the further development of the theory of smooth representations of Lie pairs.

\subsection{Formal manifolds and formal Lie groups}
We first recall some relevant notions concerning formal manifolds from \cite{CSW1}.
Let $N$ be a smooth manifold and let $k\in \BN:=\{0,1,2,\dots\}$. For every open subset $U\subset N$, write $\RC^\infty(U)$ for the $\BC$-algebra of complex-valued smooth functions on $U$, and write 
	\[
	\CO_N^{(k)}(U):=\RC^\infty(U)[[y_1, y_2, \dots, y_k]]
	\]
 for the $\BC$-algebra of formal power series with coefficients in $\RC^\infty(U)$. With the obvious restriction maps, we have a sheaf 
 \[\CO_N^{(k)}:\ U\mapsto \CO_N^{(k)}(U)\] of $\BC$-algebras. This gives a locally ringed space 
 \be\label{eq:N(k)} N^{(k)}:=(N, \CO_N^{(k)})\ee  over $\mathrm{Spec}(\BC)$. 

\begin{dfn}\label{def:formalmanifold}
	A \textbf{formal manifold} is a locally ringed space $(M, \CO_M)$ over $\mathrm{Spec}(\BC)$ such that
	\begin{itemize}
		\item the topological space $M$ is paracompact and Hausdorff; and
		\item for every $a\in M$, there is an open neighborhood $U$ of $a$ in $M$ and $n,k\in \BN$ such that $(U, \CO_M|_U)$ is isomorphic to $(\R^n)^{(k)}$ as locally ringed spaces over $\mathrm{Spec}(\BC)$.
	\end{itemize}
\end{dfn}

By abuse of notation, we will often not distinguish a formal manifold $(M, \CO_M)$ from its underlying topological space $M$, and call $\CO_M$ the structure sheaf of it. For each $a\in M$, the uniquely determined natural numbers $n$ and $k$ in Definition \ref{def:formalmanifold} are
called the dimension and the degree of $M$ at $a$,   to be denoted by
\be\label{eq:degaM} \dim_a M \qaq \deg_a M \ee respectively. An element in $\CO_M(M)$
is called a formal function on $M$.
\begin{dfn}\label{def:infinitesimalFM}
	A formal manifold is said to be infinitesimal if its underlying topological space has precisely one element.
\end{dfn}

For every formal manifold $M$, let $\mathfrak{m}_{\CO_M}$ denote the ideal of $\CO_M$ defined by
\be\label{eq:defmo}
\mathfrak{m}_{\CO_M}(U):=\{f\in \CO_M(U)\mid f_a\in \mathfrak{m}_{M,a} \textrm{ for all $a\in U$}\},
\ee
where $U$ is an open subset of $M$, $f_a$ is the germ of $f$ at $a$, and $\mathfrak{m}_{M,a}$ is the maximal ideal of the stalk $\CO_{M,a}$.
Form the quotient sheaf $\underline{\CO_M}:=\CO_M/{\mathfrak{m}_{\CO_M}}$ over $M$.
Then \be \label{eq:underM}\underline{M}:=(M,\underline{\CO_M})\ee is a smooth manifold, called the reduction of $M$.

Recall that a morphism from a formal manifold $(M_1,\CO_{M_1})$ to another formal manifold $(M_2,\CO_{M_2})$ is a pair $\varphi=(\overline\varphi, \varphi^*)$, where $\overline\varphi: M_1\rightarrow M_2$ is a continuous map and
\be\label{phin}
\varphi^*:\ \overline\varphi^{-1}\CO_{M_2}\rightarrow \CO_{M_1}
\ee
is a $\BC$-algebra sheaf homomorphism that induces local homomorphisms on the stalks.
For open subsets $U_2$ of $M_2$ and $U_1$ of $M_1$ such that $\overline\varphi(U_1)\subset U_2$,
we write
\[\varphi^*_{U_2,U_1}:\CO_{M_2}(U_2)\rightarrow \CO_{M_1}(U_1)\] for the homomorphism of $\C$-algebras induced by \eqref{phin}.
If there is no confusion, we will denote $\varphi^*_{U_2,U_1}$ by $\varphi^*_{U_2}$ or $\varphi^*$ for simplicity.
We denote by
\[
\varphi_a^*:\ \CO_{M_2,\overline{\varphi}(a)}\rightarrow \CO_{M_1,a}
\]
the induced local homomorphism on the stalks.

As finite products exist in the category of formal manifolds (see \cite[Theorem 1.9]{CSW1}), a formal Lie group is defined as a group object in that category.
The precise definition is given as follows. It is worth noting that Bochner's classical notion of formal Lie groups, as discussed in \cite{B}, corresponds to what are now commonly referred to as formal group laws. In contrast, the concept of formal Lie groups developed in this paper is more general.
\begin{dfn}\label{df:Lirgroup}
	A \textbf{formal Lie group} is a triple $(G,m,\varepsilon)$, where $G$ is a formal manifold,
    \[m: G\times G\rightarrow G \qaq \varepsilon: \mathrm{Spec}(\C)\rightarrow G \]
      are morphisms of formal manifolds, satisfying the following conditions:
	\begin{itemize}
 \item the diagram \be \label{eq:chainofG}
		\begin{CD}
			G\times G\times G @> m \times \mathrm{id}_G >>  G\times G\\
			@V\mathrm{id}_G\times m VV           @V V mV\\
			G\times G @>m>>  G \\
		\end{CD}\qquad (\textrm{$\mathrm{id}$ indicates the identity morphism})
		\ee commutes;
		\item the diagram  
		\be\label{eq:diagramofidentity}
		\xymatrix{
			\Spec(\BC)\times G \ar[dr]_= \ar[r]^{\ \ \ \  \varepsilon\times \mathrm{id}_G}        &G\times G \ar[d]_{m} & G\times \Spec(\BC)\ar[l]_{\mathrm{id}_G\times \varepsilon\ \ \ } \ar[dl]^=\\
			& G 
		}
		\ee
  commutes; and
		
		\item there is a morphism $i: G\rightarrow G$  such that
		the diagram
	    \be \label{eq:comofi}	\begin{CD}
	    	G @> \mathrm{id}_G\times i >>  G\times G@<\ i\times \mathrm{id}_G <<G\\
	    	@V VV           @Vm V V @VVV\\
	    	\Spec(\BC) @> \varepsilon >>  G @<\varepsilon<<	\Spec(\BC) \\
	    \end{CD} \ee
		commutes.
	\end{itemize}
\end{dfn}

When no confusion is possible, we will not distinguish a formal Lie group $(G, m, \varepsilon)$ from $G$.

Under the notation and assumptions of Definition \ref{df:Lirgroup},
the morphism $m: G\times G\rightarrow G$ is called the multiplication morphism, and the morphism $\varepsilon$ is called the unit morphism. Let $e$ denote the image of the unique point in $\Spec (\BC)$ under the morphism $\varepsilon:\mathrm{Spec}(\C)\rightarrow G$, to be called the identity element of $G$.
As in the case of abstract groups, a morphism $i: G\rightarrow G$ satisfying the commutative diagram \eqref{eq:comofi} is unique. The morphism $i$ satisfies $i\circ i=\mathrm{id}_G$, and it is called the inversion morphism. 

The reduction $\underline{G}$ of a formal Lie group $G$ is naturally a Lie group (see Example \ref{exam:subgroup}). 

\begin{dfn}\label{df:homoofLiegroup}
   Let $(G_1,m,\varepsilon)$ and $(G_2,m,\varepsilon)$ be two formal Lie groups. A morphism $\varphi:G_1\rightarrow G_2$ of formal manifolds is called a homomorphism between formal Lie groups if the diagram
   	\be \label{commor}
   \begin{CD}
   	G_1\times G_1 @> \varphi\times \varphi >> G_2\times G_2\\
   	@V m VV           @V V m V\\
   	G_1 @> \varphi >>  G_2
   \end{CD}
   \ee
   commutes.
\end{dfn}
As in the case of abstract groups, every homomorphism of formal Lie groups commutes with the unit morphisms and the inversion morphisms. All formal Lie groups, together with homomorphisms between them, form a category.

\begin{dfn}
	A formal Lie group is said to be infinitesimal if its underlying topological space has precisely one element.
\end{dfn}

Note that all Lie groups are formal Lie groups. On the other hand, the following example shows that every finite-dimensional complex Lie algebra determines an infinitesimal formal Lie group.

\begin{exampled}\label{exam:Lq}
    Let $\mathfrak{q}$ be a finite-dimensional complex Lie algebra. The universal enveloping algebra $\RU(\q)$ of $\q$ is a cocommutative complex Hopf algebra. By using the coalgebra structure of $\mathrm U(\mathfrak{q})$, the linear functional space $(\mathrm U(\mathfrak{q}))'$  is naturally a local $\C$-algebra which is isomorphic to a formal power series algebra. Then  the locally ringed space  $\mathcal{G}_{\mathfrak{q}}:=(\{e\},(\mathrm{U}(\mathfrak{q}))')$ is  an infinitesimal formal manifold. Furthermore, the formal manifold $\mathcal{G}_{\q}$ is an infinitesimal formal Lie group, where the multiplication morphism and the unit morphism are induced by the multiplication and the unit of $\mathrm U(\q)$ respectively.
\end{exampled}
Here and in the following, we will often not distinguish between a sheaf over a singleton and its space of global sections.

\subsection{Hopf formal algebras and Hopf formal coalgebras}\label{sec:tophopf}

In this paper, by an LCS, we mean a locally convex topological vector space over $\mathbb{C}$, which may or may not be Hausdorff. However, a complete LCS is always assumed to be Hausdorff.

For two LCS $E$ and $F$, the algebraic tensor product $E \otimes F$ admits two natural locally convex topologies: the inductive tensor product $E \otimes_{\mathrm{i}} F$ and the projective tensor product $E \otimes_{\pi} F$. These are characterized by universal properties for separately and jointly continuous bilinear maps, respectively, just as the algebraic tensor product $E \otimes F$ is characterized by bilinear maps. We denote the completions of the maximal Hausdorff quotients of these two topological tensor product spaces by
\[
E \widehat{\otimes}_{\mathrm{i}} F \quad\text{and}\quad E \widehat{\otimes}_{\pi} F,
\] 
respectively.

An \textbf{$\widehat{\otimes}_{\mathrm{i}}$-algebra} (resp.\,\textbf{ $\widehat{\otimes}_{\pi}$-algebra}) is a complete LCS $A$ together with a continuous associative multiplication $A \widehat{\otimes}_{\mathrm{i}} A \rightarrow A$ (resp.\,$A \widehat{\otimes}_{\pi} A \rightarrow A$) and a unit $\mathbb{C} \rightarrow A$. Similarly, we have the notions of \textbf{$\widehat{\otimes}_{\mathrm{i}}$-coalgebras},\textbf{ $\widehat{\otimes}_{\pi}$-coalgebras}, \textbf{Hopf $\widehat{\otimes}_{\mathrm{i}}$-algebras}, and \textbf{Hopf  $\widehat{\otimes}_{\pi}$-algebras}. See Appendix \ref{sec:topHopfalg} for 
 details.


Given a formal manifold $(M,\mathcal{O}_M)$, equip $\mathcal{O}_M(M)$ with the smooth topology (see \cite[Definition 4.1]{CSW1}). Then, together with the unit
\begin{equation}\label{eq:unitofOM}
\mathbb{C} \rightarrow \mathcal{O}_M(M),\quad c \mapsto c
\end{equation}
and the continuous multiplication
\begin{equation}\label{eq:mulofOM}
\mathcal{O}_M(M) \widehat{\otimes}_{\pi} \mathcal{O}_M(M) \rightarrow \mathcal{O}_M(M),\quad f_1 \otimes f_2 \mapsto f_1 f_2 \quad (f_1, f_2 \in \mathcal{O}_M(M)),
\end{equation}
the LCS $\mathcal{O}_M(M)$ becomes a $\widehat{\otimes}_{\pi}$-algebra (\cf \cite[Proposition 4.8]{CSW1}).

Recall from \cite[Section 5]{CSW2} the LCS $\RD_c^{-\infty}(M;\mathcal{O}_M)$ of compactly supported formal distributions on $M$. By \cite[Proposition 5.14 and Lemma A.12]{CSW2}, there are  identifications
\be \label{eq:OM'}
\RD_c^{-\infty}(M;\mathcal{O}_M) = (\mathcal{O}_M(M))' 
\ee
and
\be \label{eq:DMM=DMDM}
\RD_c^{-\infty}(M;\mathcal{O}_M) \widehat{\otimes}_{\mathrm{i}} \RD_c^{-\infty}(M;\mathcal{O}_M) = (\mathcal{O}_M(M) \widehat{\otimes}_{\pi} \mathcal{O}_M(M))'
\ee
of LCS. Then, \be\label{eq:D_ccoal} \text{the LCS $\RD_c^{-\infty}(M;\mathcal{O}_M)$ becomes an  $\widehat{\otimes}_{\mathrm{i}}$-coalgebra:}
\ee the comultiplication is  
the transpose
of \eqref{eq:mulofOM}, and the counit is the transpose
of \eqref{eq:unitofOM}.  Here and henceforth, for any LCS $E$,  $E'$ denotes the space of all continuous linear functionals on $E$,  equipped with the strong topology.


We make the following definitions.

\begin{dfn}\label{def:foralg}

\noindent(a) A  $\widehat{\otimes}_{\pi}$-algebra is called a \textbf{formal algebra} if it is isomorphic to $\mathcal{O}_M(M)$ as  $\widehat{\otimes}_{\pi}$-algebras for some formal manifold $M$ (cf. \cite[Definition 5.6]{CSW1}).
    
\noindent(b) An  $\widehat{\otimes}_{\mathrm{i}}$-coalgebra is called a \textbf{formal coalgebra} if it is isomorphic to $\RD_c^{-\infty}(M;\mathcal{O}_M)$ as  $\widehat{\otimes}_{\mathrm{i}}$-coalgebras for some formal manifold $M$.
\end{dfn}
Together with all continuous algebra (resp.\,coalgebra) homomorphisms, all formal algebras (resp.\,coalgebras) form a category. 
\begin{dfn}\label{def:formalHopfalg}

\noindent(a) A Hopf $\widehat{\otimes}_{\pi}$-algebra is called a \textbf{Hopf formal algebra} if it is a formal algebra as a $\widehat{\otimes}_{\pi}$-algebra.
    
\noindent(b) A Hopf $\widehat{\otimes}_{\mathrm{i}}$-algebra is called a \textbf{Hopf formal coalgebra} if it is a formal coalgebra as an $\widehat{\otimes}_{\mathrm{i}}$-coalgebra.
\end{dfn}

Together with all continuous Hopf algebra homomorphisms, all Hopf formal algebras form a category, and all Hopf formal coalgebras form a category.

\subsection{Formal Lie theory theorem}
The key result of this paper may be summarized as follows.
\begin{thmd}\label{thm:eqFP}
	The following categories are equivalent to each other:
	\begin{enumerate}
		\item the category of Lie pairs;
		\item the category of Hopf formal coalgebras;
		\item the opposite category of the category of Hopf formal algebras;
		\item the category of formal Lie groups.
	\end{enumerate}
\end{thmd}
For every formal Lie group $G $, the space $\CO_G(G)$ is a Hopf formal algebra (see Proposition \ref{prop:formalliegrouptohopfalg}). By taking the continuous dual, the space $\RD_c^{-\infty}(G;\CO_G)=(\CO_G(G))'$ is a Hopf formal coalgebra (see Proposition \ref{prop:formalliegrouptohopfalg}). Finally, together with the adjoint representation $\mathrm{Ad}:\underline G \curvearrowright \mathfrak{g}$ and the canonical Lie algebra embedding $\underline{\mathfrak{g}}\rightarrow \mathfrak{g}$, the pair $(\mathfrak{g},\underline{G})$ is a Lie pair (see Proposition \ref{prop:GtoLiepair}), where $\mathfrak{g}$ is the complex Lie algebra of $G$ (see Definition \ref{def:Liealg}), and $\underline{\mathfrak{g}}$ is the complexified  Lie algebra of $\underline{G}$. These establish the equivalence between the categories in Theorem \ref{thm:eqFP} (see Section \ref{sec:proofofmain}).

Observe that the category of Lie groups forms a full subcategory of the category of formal Lie groups. When restricted to this subcategory, Theorem \ref{thm:eqFP} recovers a well-known characterization: a real Lie group $L$ is determined by certain topological Hopf algebras—namely, the space of compactly supported distributions on $L$ or the space of smooth functions on $L$.

Similarly, the category of finite-dimensional complex Lie algebras embeds as a full subcategory of the category of formal Lie groups. Restricting to this subcategory, Theorem \ref{thm:eqFP} specializes to a version of the formal Lie theory theorem, which can be stated as follows (see \cite[(14.2.3)]{Ha} and \cite[Theorem 5.18]{MM}).

\begin{thmd}\label{thm:eqFPpre}
The following categories are equivalent to each other:
\begin{enumerate}
	\item the category of finite-dimensional complex Lie algebras;
	\item the category of Hopf algebras over $\mathbb{C}$ that, as associative $\mathbb{C}$-algebras, are generated by finitely many primitive elements;
	\item the opposite category of the category of Hopf $\wh\otimes_{\pi}$-algebras that, as $\wh\otimes_{\pi}$-algebras, are isomorphic to   formal power series algebras;
	\item the category of infinitesimal formal Lie groups.
\end{enumerate}
\end{thmd}
See \eqref{eq:Prim} for the definition of primitive elements in a complex Hopf algebra. Here  formal power series algebras carry the usual term-wise convergence topologies. 
It is clear that the category of those  Hopf $\wh\otimes_{\pi}$-algebras of (3) is equivalent to the category of finite-dimensional formal group laws over $\BC$ defined in \cite[p.~xvii]{Ha}. 

In Theorem \ref{thm:eqFPpre}, the equivalence of (1) and (2) is given by the functors \[\q\mapsto\RU(\q) \  \text{($\q$ is an object of (1))},\qaq A\mapsto \mathrm{Prim}(A)\ \text{($A$ is an object of (2))},\]where $\mathrm{Prim}(A)$ is the set of primitive elements of $A$ (see \eqref{eq:Prim}). The equivalence of (2) and (3) is given by the functors \[A\mapsto A' \quad \text{($A$ is an object of (2))},\qaq C\mapsto C'\quad \text{($C$ is an object of (3))},\] where $A$ is equipped with the finest locally convex topology, and $C'$ is an object of (2) as an abstract Hopf algebra. The equivalence of (3) and (4) is given by the functors \[C\mapsto (\{e\},C)\ \text{($C$ is an object of (3))},\qaq G\mapsto \CO_G(G) \ \text{($G$ is an object of (4))}. \]



The organization of the paper is as follows. Section 2 recalls tangent vectors on formal manifolds and introduces the complex Lie algebras associated with formal Lie groups. We then move on to Section 3, which discusses formal Lie subgroups and their basic properties. In Section 4, we study formal actions of formal Lie groups on formal manifolds, with an emphasis on properties related to stabilizer formal Lie subgroups and fixed points. 
Section 5 introduces isotropy representations in the context of formal Lie groups and constructs a functor from the category of formal Lie groups to that of Lie pairs. Section 6 examines semidirect products and normal formal Lie subgroups. Section 7 is devoted to quotients arising from formal actions, particularly quotients of formal Lie groups.
Section 8 constructs a formal Lie group associated to a Lie pair. 
Section 9 presents the proof of Theorem \ref{thm:eqFP}. Finally, Appendix A provides the basic definitions concerning topological tensor products and topological Hopf algebras used throughout the paper.

\section{Complex Lie algebras of formal Lie groups}
\subsection{Complex tangent spaces}
Let $(M,\CO_M)$ be a formal manifold and $a\in M$. 
Recall from \cite[Section 2.3]{CSW1} that the formal stalk at $a$ is the local $\BC$-algebra
\be\label{eq:dualofpower} \wh\CO_{M,a}:=\varprojlim_{r\in \BN} \CO_{M,a}/\m_{M,a}^r,\ee where $\m_{M,a}$ is the maximal ideal of the stalk  $\CO_{M,a}$. Equip $\wh\CO_{M,a}$ with the projective limit topology. Then  \be \label{eq:whO=series}\wh\CO_{M,a}\cong\BC[[x_1,x_2,\dots,x_n,y_1,y_2,\dots,y_k]] \quad (n:=\dim_a M, k:=\deg_a M)\ee 
as $\wh\otimes_{\pi}$-algebras (see \cite[Proposition 2.6]{CSW1}), where the formal power series algebra carries the usual term-wise convergence topology.

Recall from \cite[Section 5.3]{CSW2} the space $\mathrm{Dist}_a(M)$ of formal distributions on $M$ supported at $a$, which is equipped with the subspace topology of $\RD^{-\infty}_c(M;\CO_M)$. By \cite[Lemma 5.26 and Lemma A.12]{CSW2},
there are identifications  \be\label{eq:Dist'} \mathrm{Dist}_a(M)=(\wh\CO_{M,a})'\ee  and  \be\label{eq:DistDist'} \mathrm{Dist}_a(M)\wh\otimes_{\mathrm{i}}\mathrm{Dist}_a(M)=(\wh\CO_{M,a}\wh\otimes_\pi\wh\CO_{M,a})'\ee of LCS. 
Then the space $\mathrm{Dist}_a(M)$ becomes an
$\wh\otimes_{\mathrm{i}}$-coalgebra: the comultiplication is the transpose of
the multiplication of $\wh\CO_{M,a}$, and the
counit is the transpose of the unit of $\wh\CO_{M,a}$. Furthermore, by
\eqref{eq:whO=series},
\be\label{eq:Disttopoly}
\mathrm{Dist}_a(M)\cong
\C[x_1^*,x_2^*,\dots,x_n^*,y_1^*,y_2^*,\dots,y_k^*]
\ee
as $\wh\otimes_{\mathrm{i}}$-coalgebras, where $\C[x_1^*,x_2^*,\dots,x_n^*,y_1^*,y^*_2,\dots,y_k^*]$ is equipped with 
the finest locally convex topology (see \cite[Example B.3]{CSW2}). 


We make the following definition (see \cite[Definition 2.1 and Lemma 2.5]{CSW4}). 
\begin{dfn}\label{def:tangent}
    A \textbf{tangent vector}  of $M$ at $a$ is a linear functional
$\eta :\wh\CO_{M,a}\rightarrow \BC$ such that 
\be\label{eq:tangent} \eta(f_1f_2)=f_1(a)\cdot \la \eta,f_2\ra+\la\eta,f_1\ra\cdot f_2(a)\quad \text{for all }f_1,f_2\in \wh\CO_{M,a}.\ee
\end{dfn}
 Here, for every $f\in \wh\CO_{M,a}$, $f(a)$ denotes its image  under 
the quotient map  
\[\delta_a:\ \wh\CO_{M,a}\rightarrow \wh\CO_{M,a}/\wh\m_{M,a}=\C\quad(\text{$\wh\m_{M,a}$ is the maximal ideal of $\wh\CO_{M,a}$}).\] 
Write $\mathrm{T}_a(M)$ for the space of all tangent vectors of $M$ at $a$, to be called the complex tangent space of $M$ at $a$. 
Recall from  \cite[Lemma 2.2]{CSW4} that $\RT_a(M)$ is finite-dimensional, and we endow it with the Euclidean topology.
As usual, for each $\eta\in \mathrm{T}_a(M)$,  we also view it as a linear functional on $\CO_M(M)$ via the composition
\[
\CO_M(M)\rightarrow \CO_{M,a}\rightarrow \wh\CO_{M,a}\xrightarrow{\eta}\BC.\]
It is clear that  \[ \delta_a\in \mathrm{Dist}_a(M)\qaq\mathrm{T}_a(M)\subset\mathrm{Dist}_a(M).\]  

Let $\varphi=(\overline{\varphi},\varphi^*): (M,\CO_M)\rightarrow (M', \CO_{M'}) $ be a morphism between formal manifolds. Since  the composite homomorphism \[\CO_M(M)\rightarrow \CO_{M,a}\rightarrow \wh{\CO}_{M,a}\] is surjective (see \cite[Corollary 2.12]{CSW1}), there exists a unique  map 
\be \label{eq:moronformalstalk}
\varphi_a^*:\ \wh\CO_{M',\overline\varphi(a)}\rightarrow \wh\CO_{M,a}
\ee 
such that the diagram
\be \label{eq:commformalstalk} 
\begin{CD}
    \CO_{M'}(M')@> \varphi^*>> \CO_M(M) \\ @VVV @VVV \\ \wh{\CO}_{M',\overline\varphi(a)} @>\varphi^*_a>> \wh\CO_{M,a} 
\end{CD}
\ee
commutes. Moreover, the map \eqref{eq:moronformalstalk} is a continuous local homomorphism. 
By taking the transpose of 
 \eqref{eq:moronformalstalk}, we obtain  a continuous coalgebra homomorphism
\be\label{eq:{}t}
{}^{t}\varphi_a^*:\ \mathrm{Dist}_a(M)\rightarrow \mathrm{Dist}_{\overline\varphi(a)}(M'),
\ee
as well as  a linear map
\be\label{eq:diffonT}
d\varphi_a\,:=\,{}^{t}\varphi_a^*|_{T_a(M)}:\ \mathrm{T}_a(M)\rightarrow \mathrm{T}_{\overline\varphi(a)}(M'). 
\ee The map $d\varphi_a$ is called the differential of $\varphi$ at $a$.
It is obvious that
\be\label{eq:chainruleofdvarphi}   {}^{t}(\varphi'\circ\varphi)_a^* = {}^{t}(\varphi')_{\overline{\varphi}(a)}^* \circ {}^{t}\varphi_a^*\qaq d(\varphi'\circ \varphi)_a=d\varphi'_{\overline{\varphi}(a)}\circ d\varphi_a,\ee 
where $\varphi': M'\rightarrow M_1$ is another morphism of formal manifolds.

In what follows, we present some basic results on complex tangent spaces and differentials that will be used throughout the paper. 
\begin{lemd}\label{lem:T3}
    Let $M_1$, $M_2$ be two formal manifolds, and let $a_1\in M_1$, $a_2\in M_2$.
Then there are identifications
\begin{eqnarray*}
    \mathrm{Dist}_{(a_1,a_2)}(M_1\times M_2)&=&\mathrm{Dist}_{a_1}(M_1)\wh\otimes_{\Ri} \mathrm{Dist}_{a_2}(M_2)\\ &=&
\mathrm{Dist}_{a_1}(M_1)\otimes_{\Ri} \mathrm{Dist}_{a_2}(M_2)
\end{eqnarray*}
of LCS. Under the composite identification, we have that
$\delta_{(a_1,a_2)}=\delta_{a_1}\otimes\delta_{a_2}$ and 
\begin{eqnarray}
\label{eq:T=T+T}\mathrm{T}_{(a_1,a_2)} (M_1\times M_2)&=&(\mathrm{T}_{a_1}(M_1)\otimes\delta_{a_2})\oplus (\delta_{a_1}\otimes\mathrm{T}_{a_2}(M_2))\\
\nonumber  &= &\mathrm{T}_{a_1}(M_1)\oplus \mathrm{T}_{a_2}(M_2)
\end{eqnarray} as complex vector spaces. 
\end{lemd}

\begin{proof}
Recall from \cite[Theorem 1.9]{CSW1} that  \be\label{eq:O3} \wh\CO_{M_1\times M_2,(a_1,a_2)}=\wh\CO_{M_1,a_1}\wh\otimes_{\pi}\wh\CO_{M_2,a_2}\ee as LCS. Taking the continuous duals and using \eqref{eq:Dist'} and \eqref{eq:DistDist'}, we obtain that 
\begin{eqnarray*}
&&\mathrm{Dist}_{(a_1,a_2)}(M_1\times M_2)=(\wh\CO_{M_1\times M_2,(a_1,a_2)})'\\ &=&(\wh\CO_{M_1,a_1}\wh\otimes_{\pi}\wh\CO_{M_2,a_2})'
=\mathrm{Dist}_{a_1}(M_1)\wh\otimes_{\Ri} \mathrm{Dist}_{a_2}(M_2)
\end{eqnarray*} as LCS, which proves the first identification. The second one follows from \eqref{eq:Disttopoly}. 

Under the composite identification, the equality
\(\delta_{(a_1,a_2)}=\delta_{a_1}\otimes \delta_{a_2}\)
is immediate, while the assertion \eqref{eq:T=T+T} is straightforward by using \cite[Lemma 2.2]{CSW4}. \end{proof}

Given two morphisms $\varphi_1: M\rightarrow M_1$ and $\varphi_2: M\rightarrow M_2$ of formal manifolds, when no confusion is possible, we still use  
   \be\label{eq:varphi1times2} \varphi_1\times \varphi_2:\ M\rightarrow M_1\times M_2\ee to denote the composition
   \[
   M\xrightarrow{\Delta_{M}}M\times M\xrightarrow{\varphi_1\times\varphi_2}M_1\times M_2. \]
   Here and henceforth, $\Delta_M$ denotes the diagonal morphism of $M$. 
  The following lemma is routine to verify.  
\begin{lemd}\label{lem:difftimes}Let $\varphi_1=(\overline{\varphi_1},\varphi^*_1): M\rightarrow M_1$ and $\varphi_2=(\overline{\varphi_2},\varphi^*_2):M\rightarrow M_2$ be two morphisms of formal manifolds. 
	Then the differential $d(\varphi_1\times\varphi_2)_a$ of the morphism \[\varphi_1\times \varphi_2 :\ M\rightarrow M_1\times M_2\]  at $a\in M$ is given by 
\[
\begin{aligned}
\RT_a(M) &\longrightarrow \mathrm{T}_{(\overline{\varphi_1}(a),\overline{\varphi_2}(a))} (M_1\times M_2)=(\mathrm{T}_{\overline{\varphi_1}(a)}(M_1)\otimes\delta_{\overline{\varphi_2}(a)})\oplus (\delta_{\overline{\varphi_1}(a)}\otimes\mathrm{T}_{\overline{\varphi_2}(a)}(M_2)),\\ 
\eta &\longmapsto d(\varphi_1)_a(\eta)\otimes \delta_{\overline{\varphi_2}(a)}
        + \delta_{\overline{\varphi_1}(a)}\otimes d(\varphi_2)_a(\eta).
\end{aligned}
\] 
\end{lemd}


\subsection{Complex Lie algebras of formal Lie groups} 
Let $G$ be a formal Lie group. 
Then the morphisms $m$, $\varepsilon$, and $i$ as in Definition \ref{df:Lirgroup} induce respectively the  continuous algebra homomorphisms
    \be\label{eq:co-hopfalg}
\begin{aligned}
	m^*_{(e,e)}:\wh\CO_{G,e}&\longrightarrow \wh\CO_{G\times G,(e,e)}
	=\wh\CO_{G,e}\wh{\otimes}_{\pi}\wh\CO_{G,e} \quad(\text{see \eqref{eq:O3}}),\\
	\varepsilon^*_e:\wh\CO_{G,e}&\longrightarrow \BC,
	\qquad\text{and}\qquad
	i^*_e:\wh\CO_{G,e}\longrightarrow \wh\CO_{G,e}.
\end{aligned}
\ee

A direct application of the commutative diagram \eqref{eq:commformalstalk} together with those in Definitions \ref{df:Lirgroup} and \ref{df:homoofLiegroup}
yields the following result. 

\begin{lemd}\label{lem:OeHopf}

\noindent(a)
The $\wh\otimes_{\pi}$-algebra $\wh\CO_{G,e}$, together with the comultiplication $m^*_{(e,e)}$, the counit $\varepsilon^*_e$ and the antipode $i_e^*$, 
forms a Hopf $\wh\otimes_{\pi}$-algebra. 

 \noindent(b)  Let $\varphi: G_1\rightarrow G_2$ be a homomorphism of formal Lie groups.
    Then the continuous algebra homomorphism $\varphi_e^*: \wh\CO_{G_2,e}\rightarrow \wh\CO_{G_1,e}$ is a homomorphism between Hopf $\wh\otimes_{\pi}$-algebras.
\end{lemd}



The following proposition is a direct consequence of Lemma \ref{lem:OeHopf}.
\begin{prpd}\label{prop:DistasHopf} Let $G$ be a formal Lie group.
 
  \noindent(a)  Together with the multiplication  \be\label{eq:tm} {}^{t} m_{(e,e)}^*:\ \mathrm{Dist}_e(G)\wh\otimes_{\Ri}\mathrm{Dist}_e(G)=\mathrm{Dist}_{e}(G)\otimes_{\Ri} \mathrm{Dist}_{e}(G)\rightarrow \mathrm{Dist}_e(G),\ee the unit  \be\label{eq:1todeltae} {}^{t}\varepsilon_e^*:\ 
    \BC\rightarrow \mathrm{Dist}_e(G),\quad c\mapsto c\cdot \delta_e, 
    \ee and the antipode  ${}^{t} i_e^*: \mathrm{Dist}_e(G)\rightarrow \mathrm{Dist}_e(G)$, the $\wh\otimes_\mathrm{i}$-coalgebra $\mathrm{Dist}_e(G)$ becomes a Hopf $\wh\otimes_\mathrm{i}$-algebra. 

    \noindent(b)  Let $\varphi: G_1\rightarrow G_2$ be a homomorphism between formal Lie groups. Then the continuous coalgebra homomorphism \[{}^t\varphi^*_e: \ \mathrm{Dist}_e(G_1)\rightarrow \mathrm{Dist}_{e}(G_2) \] is a homomorphism between Hopf $\wh\otimes_\mathrm{i}$-algebras.
\end{prpd}

For a complex Hopf algebra $A$ with the comultiplication $\Delta$,  we write $\mathrm{Prim}(A)$ for the set of primitive elements of $A$, namely,
\be \label{eq:Prim}
\mathrm{Prim}(A):=\{\eta\in A:\ \Delta(\eta)=\eta\otimes 1+1\otimes \eta\}. 
\ee It is well-known that $\mathrm{Prim}(A)$ forms a Lie algebra under the commutator bracket. 

By \eqref{eq:tm}, $\mathrm{Dist}_e(G)$ is an abstract  Hopf algebra over $\BC$ (forgetting the topology on $\mathrm{Dist}_e(G)$). Then, \be \label{eq:T=prim}\mathrm{T}_e(G)=\mathrm{Prim}(\mathrm{Dist}_e(G))\ee forms a finite-dimensional complex Lie algebra with the Lie bracket
     \[[\eta_1,\eta_2]:=\eta_1\eta_2-\eta_2\eta_1\quad (\eta_1,\eta_2\in \mathrm{T}_e(G)).\] 
Here and in the following, for $\eta_1,\eta_2\in \mathrm{Dist}_e(G)$, set 
\be \eta_1\eta_2\,:=\, {}^tm^*_{(e,e)}(\eta_1\otimes\eta_2)\in \mathrm{Dist}_e(G). \ee


 We make the following definition.
\begin{dfn}\label{def:Liealg}
    The Lie algebra  $\mathrm{T}_e(G)$ is called the \textbf{complex Lie algebra} of the formal Lie group $G$.
\end{dfn}     
Here and in the following, 
the complex Lie algebra of the formal Lie group $G$ is denoted by $\Lie_{\BC}(G)$, or simply by $\g$ if there is no confusion.  
When $G$ is a Lie group,  $\Lie_{\BC}(G)$ is just the complexified Lie algebra of $G$. In this case, we also write $\Lie(G)$ for the real Lie algebra of $G$.  

By using Proposition \ref{prop:DistasHopf}, Lemma \ref{lem:difftimes} and the commutative diagram \eqref{eq:comofi}, we obtain the following result.
\begin{lemd}
	\label{lem:inverse}
	(a) The differential $dm_{(e,e)}$ of the multiplication morphism $m:G\times G\rightarrow G $ at $(e,e)\in G\times G$ coincides with the map  \[\g\oplus \g\rightarrow \g ,\quad (\eta_1,\eta_2)\mapsto\eta_1+\eta_2.\]
    
	\noindent (b) The differential $di_e$ of the inversion morphism $i:G\rightarrow G$ at $e\in G$ is \[ \g\rightarrow \g,\quad \eta\mapsto -\eta.\]  
\end{lemd}	

As usual, $\RU(\g)$ is a Hopf algebra. Equip $\RU(\g)$ with the canonical topology, namely, the finest locally convex topology.  Then $\RU(\g)$ becomes a Hopf $\wh\otimes_\mathrm{i}$-algebra.
\begin{prpd}\label{prop:Liealg} Let $G$ be a formal Lie group.

      \noindent (a) The algebra homomorphism \be \label{eq:Utodist}\RU(\g)\rightarrow  \mathrm{Dist}_e(G)\ee induced by the inclusion $\g\rightarrow \mathrm{Dist}_e(G)$
      is an isomorphism between Hopf $\wh\otimes_\mathrm{i}$-algebras. 
      \noindent (b) Let $\varphi: G_1\rightarrow G_2$ be a homomorphism between formal Lie groups. Then 
      the differential
      \[d\varphi_e: \ \Lie_{\BC}(G_1) \rightarrow \Lie_{\BC}(G_2) \]  is a homomorphism between Lie algebras.
\end{prpd}
\begin{proof}
     By Lemma \ref{lem:OeHopf} and \eqref{eq:whO=series}, the assertion (a) follows from the formal Lie theory theorem (see \cite[(14.2.3)]{Ha} and \cite[Theorem 5.18]{MM} or Theorem \ref{thm:eqFPpre}). 
     The assertion (b) is clear by using Proposition \ref{prop:DistasHopf}. 
\end{proof}

\begin{exampled}\label{exam:LieofLq}
 For every finite-dimensional complex Lie algebra $\q$, there is an identification 
    $\Lie_{\BC}(\mathcal{G}_{\q})=\q $
    between Lie algebras (see Example \ref{exam:Lq}).     
\end{exampled}
 



\section{Formal Lie  subgroups}
 \subsection{Formal Lie groups as functors}\label{sec:asfunc}
As in the case of algebraic groups (see  \cite[1.4]{Mil}),   a formal Lie group can be regarded as a functor from the opposite category $\mathcal{F}\mathrm{M}^{\mathrm{op}}$ of the category $\mathcal{F}\mathrm{M}$ of formal manifolds to the category $\mathcal{G}\mathrm{rp}$ of groups.

For every  formal manifold $N$, we have a functor 
\be\label{eq:N(M)} h_N:\  \mathcal{F}\mathrm{M}^{\mathrm{op}} \rightarrow \mathcal{S}\mathrm{et},\quad M\mapsto N(M)\,:=\,\mathrm{Mor}(M,N),\ee
where $\mathcal{S}\mathrm{et}$ is  the category of sets, and $\mathrm{Mor}(M,N)$ denotes the set of all morphisms $M\rightarrow N$ between formal manifolds.  
For every morphism $\varphi:N_1\rightarrow N_2$ of formal manifolds and every formal manifold $M$, we write $\varphi_M$ for the map
\be\label{eq:varphiM}
\varphi_M:\ N_1(M)\rightarrow N_2(M),\quad \psi\mapsto \varphi\circ \psi.
\ee
Then, for each formal Lie group $G$ and each formal manifold $M$, the set $G(M)$, together with the maps $m_M, \varepsilon_M$ and $i_M$, forms a group. 
Accordingly, $h_G$ can be viewed as a functor $\mathcal{F}\mathrm{M}^{\mathrm{op}}\rightarrow \mathcal{G}\mathrm{rp}$.
The following lemma follows immediately from the Yoneda lemma (see \cite[p.~61]{Sa}), and will be used frequently throughout the paper.
    \begin{lemd} \label{prop:asfunctors}
\noindent  (a)
The functor \[ \mathcal{F}\mathrm{M}\rightarrow \mathcal{F}\mathrm{un}(\mathcal{F}\mathrm{M}^{\mathrm{op}},\mathcal{S}\mathrm{et}),\quad N \mapsto  h_N \] is fully faithful, where $\mathcal{F}\mathrm{un}(\mathcal{F}\mathrm{M}^{\mathrm{op}},\mathcal{S}\mathrm{et})$ denotes the category of functors from $\mathcal{F}\mathrm{M}^{\mathrm{op}}$ to $\mathcal{S}\mathrm{et}$.

\noindent (b) The functor \[ \mathcal{F}\mathrm{LG}\rightarrow \mathcal{F}\mathrm{un}(\mathcal{F}\mathrm{M}^{\mathrm{op}},\mathcal{G}\mathrm{rp}),\quad G \mapsto  h_G\] 
is fully faithful, where $\mathcal{F}\mathrm{LG}$ is the category of formal Lie groups, and $\mathcal{F}\mathrm{un}(\mathcal{F}\mathrm{M}^{\mathrm{op}},\mathcal{G}\mathrm{rp})$ denotes the category of functors from $\mathcal{F}\mathrm{M}^{\mathrm{op}}$ to $\mathcal{G}\mathrm{rp}$. 

\noindent (c) For every formal manifold $N$ and every functor $F: \mathcal{F}\mathrm{M}^{\mathrm{op}}\rightarrow \mathcal{G}\mathrm{rp}$ such that the functor $h_N:\mathcal{F}\mathrm{M}^{\mathrm{op}}\rightarrow \mathcal S\mathrm{et}$ equals the composition 
\[
\mathcal{F}\mathrm{M}^{\mathrm{op}}\xrightarrow{F} \mathcal{G}\mathrm{rp}\xrightarrow{\textrm{the forgetful functor}} \mathcal S\mathrm{et},
\]
there is a unique formal Lie group structure on the formal manifold $N$ such that $h_N=F$ as functors from $\mathcal{F}\mathrm{M}^{\mathrm{op}}$ to  $\mathcal{G}\mathrm{rp}$.

\end{lemd}
\subsection{Formal submanifolds}
Here we recall some notions concerning formal submanifolds from \cite{CSW4}.
\begin{dfn}\label{def:immsim}
 \noindent (a) An \textbf{immersed formal submanifold} of a formal manifold  $(M,\CO_{M})$ is a pair  $(S,(j^{-1}\CO_M)/\CI)$ consisting of
\begin{itemize} 

\item a subset $S$ of $M$, equipped with a topology such that the inclusion map $j:S\to M$ is continuous; and

\item a quotient sheaf $(j^{-1}\CO_M)/\CI$ over  $S$ such that all stalks of $(j^{-1}\CO_M)/\CI$ are local rings, where $\CI$ is an ideal  of the sheaf $j^{-1}\CO_M$ of $\BC$-algebras over $S$,  
\end{itemize} subject to the following condition: the locally ringed space  $(S,(j^{-1}\CO_M)/\CI)$ over $\mathrm{Spec}(\C)$ is a formal manifold. 

\noindent (b) An immersed formal submanifold $(S,(j^{-1}\CO_M)/\CI)$ of $M$ is called an \textbf{embedded formal submanifold} if the topology of $S$ agrees with the subspace topology of $M$.


\noindent (c) An embedded  formal submanifold $(S,(j^{-1}\CO_M)/\CI)$ of $M$ is called a \textbf{closed formal submanifold} if $S$ is closed in $M$. 



\end{dfn} 
For each immersed formal submanifold $(S,(j^{-1}\CO_M)/\CI)$ of $(M,\CO_M)$,  there is a canonical morphism \be\label{eq:mapiota} \iota=(\overline\iota,\iota^*):\ (S,(j^{-1}\CO_M)/\CI)\rightarrow (M,\CO_M)\ee of formal manifolds, where $\overline{\iota}=j$ is the inclusion map, and for every open subset $U$ of $M$, $\iota^*_{U}$ is the composition map \be \label{eq:iotaU}\iota^*_{U}:\ \CO_M(U)\rightarrow (j^{-1}\CO_M)(j^{-1}(U))\rightarrow ((j^{-1}\CO_M)/\CI)(j^{-1}(U)).\ee 
We call \eqref{eq:mapiota} the inclusion morphism.

By abuse of notation, we will often not distinguish an immersed formal submanifold $(S,(j^{-1}\CO_M)/\CI)$  of $(M,\CO_M)$  
with its underlying topological space $S$, and call $\CI$ its defining ideal.


Let $\mathrm{Imm}(M)$ denote the family of all pairs $(M',\varphi)$, where $M'$ is a formal manifold and $\varphi$ is an injective immersion from $M'$ to $M$. 
As usual,  a morphism $\varphi=(\overline\varphi, \varphi^*): M'\rightarrow M$ is called an immersion if the differential $d\varphi_{b}$ (see \eqref{eq:diffonT}) is injective for every $b\in M'$, and $\varphi$ is said to be injective 
if  $\overline\varphi$ is injective.

By \eqref{eq:iotaU}, it is clear that for each immersed formal submanifold $S$ of $M$, 
the pair $(S,\iota)$ is in $\mathrm{Imm}(M)$. 
Conversely, for each pair $(M',\varphi)\in \mathrm{Imm}(M)$, there exists a unique immersed formal submanifold $S$ of $M$ and a unique isomorphism $\varphi': M'\rightarrow S$
of formal manifolds such that the diagram 
\be \label{eq:M'toS}\xymatrix{&S\ar[d]^{\iota}\\M'\ar[ru]^{\varphi'}\ar[r]^{\varphi}&M}\ee commutes (see \cite[Proposition 4.10]{CSW4}).


For each $(M',\varphi)\in\mathrm{Imm}(M)$, 
recall from \cite[Proposition 4.9]{CSW4} that $\varphi$ is a monomorphism in the category of formal manifolds. It follows that, for each formal manifold $N$ and each immersed formal submanifold $S$ of $M$, 
\be\label{inj00}
\textrm{the map $\iota_N: S(N)\rightarrow M(N)$ is injective.}
\ee

The following lemma is straightforward by using charts of infinitesimal formal manifolds (see \cite[p.~2]{CSW4}).
\begin{lemd}\label{prop:subLs} Assume that $M$ is  an infinitesimal formal manifold whose underlying set is $\{a\}$.
    Let $E$ be a $\BC$-linear subspace of $\RT_a(M)$. Then there exists an immersed formal submanifold $S_E$ of $M$ with underlying set $\{a\}$ such that $E$ equals 
    the image of the differential $d\iota_a: \mathrm{T}_a(S_E)\rightarrow \RT_a(M)$.
\end{lemd}

\subsection{Formal Lie subgroups}
Let $G$ be a formal Lie group. 
\begin{dfn}\label{def:subgroup} An immersed formal submanifold  $H$ of $G$ is called a \textbf{formal Lie subgroup} if
there exist morphisms $m_H: H\times H\rightarrow H$, $\varepsilon_H: \Spec(\BC)\rightarrow H$, and $i_H: H\rightarrow H$ such that the three diagrams 
\be\label{eq:diaofsubgroup} \begin{CD}
    H\times H @> m_H >>H \\ @V\iota\times \iota VV @VV\iota V \\G\times G @>m>> G,
\end{CD}\   \quad\quad \begin{CD}
     \Spec(\BC) @>\varepsilon_{H} >> H \\ @V= VV @V V \iota V \\\Spec(\BC) @>\varepsilon>> G,
\end{CD}\ \ 
\qaq \ \  \begin{CD}
     H @> i_H >>H \\ @V\iota VV @V V \iota V \\G @>i>> G
\end{CD}\ee
commute. A formal Lie subgroup $H$ of $G$ is said to be closed if it is a closed formal submanifold of  $G$.
\end{dfn}

It follows from \eqref{inj00} that the morphisms $m_H$, $\varepsilon_H$ and $i_H$ satisfying the commutative diagrams in \eqref{eq:diaofsubgroup}
are unique.
Together with the multiplication morphism $m_H$ and the unit morphism $\varepsilon_H$, the formal manifold $H$ becomes a formal Lie group. Moreover, the inclusion morphism $\iota: H\rightarrow G$ is a homomorphism between formal Lie groups. 
By \cite[Theorem 3.21]{War}, the formal Lie subgroup $H$ is closed as soon as it is an embedded formal submanifold of $G$. 
Additionally, by Proposition \ref{prop:Liealg},
 the differential $d\iota_e: \h\rightarrow \g$ of  $\iota: H\rightarrow G$ at $e$
 is an injective Lie algebra homomorphism. In view of this, we 
 will often identify $\h$ with the subalgebra $d\iota_e(\h)$ of $\g$ without further explanation. 

For every immersed formal submanifold $S$ of $G$,  it follows from Lemma \ref{prop:asfunctors} that $S$ is a formal Lie subgroup of $G$ if and only if for each formal manifold $N$, the image of the injective map 
$\iota_N: S(N)\rightarrow G(N)$ is a subgroup of $G(N)$ (see \eqref{inj00}).

Let $\mathrm{Imh}(G)$ denote the family of all pairs $(G',\varphi)$ such that $G'$ is a formal Lie group and $\varphi$ is a homomorphism between formal Lie groups that is an injective immersion as a morphism between formal manifolds. It is clear that for each formal Lie subgroup $H$ of $G$, $(H,\iota)$ is in $\mathrm{Imh}(G)$. Conversely, for each pair $(G',\varphi)$ in $\mathrm{Imh}(G)$, 
there exists a unique formal Lie subgroup $H$ of $G$ and a unique isomorphism $\varphi': G'\rightarrow H$
of formal Lie groups such that the diagram 
\be \label{eq:StoG}\xymatrix{&H\ar[d]^{\iota}\\G'\ar[ru]^{\varphi'}\ar[r]^{\varphi}&G}\ee commutes (\cf \cite[Proposition 4.10]{CSW4}).




\subsection{Examples of formal Lie subgroups}

Let $M$ be a formal manifold, and let $a\in M$.
Since the natural map \[\CO_{M,a}/\cap_{r\in \BN} \m_{M,a}^r\rightarrow \wh{\CO}_{M,a}\] is an isomorphism of $\BC$-algebras (see \cite[Proposition 2.6]{CSW1}), the locally ringed space \be  \label{eq:M_a}M_{\{a\}}:=(\{a\},\CO_{M,a}/\cap_{r\in \BN} \m_{M,a}^r)\ee is a closed formal submanifold of $M$. 
On the other hand, it is clear that the reduction \[ \underline{M}:=(M,\underline{\CO_M})\] is a closed formal submanifold of $M$. 

Let $\varphi: M\rightarrow M'$ be a morphism of formal manifolds. 
Then there are unique morphisms 
\be\label{eq:reductionvarphi}
\underline \varphi:\ \underline{M}\rightarrow \underline{M'}\qquad\text{and}\qquad \varphi_{a}:\ M_{\{a\}}\rightarrow M'_{\{\overline{\varphi}(a)\}}\ee   such that the diagrams
\be \label{eq:underlinecommdiag}
\begin{CD}
	M @>  \varphi >> M'\\
	@A\iota AA          @AA\iota A\\
	\underline{M}@> \underline \varphi >>  \underline{M'} \\
\end{CD} \qquad  \text{and}\qquad \begin{CD}
	M @>  \varphi >> M'\\
	@A \iota AA          @AA \iota A\\
	M_{\{a\}}@> \varphi_{a} >> M'_{\{\overline{\varphi}(a)\}} \\
\end{CD}
\ee
commute, respectively. Moreover, the morphism $\underline \varphi$ is a smooth map between smooth manifolds.

Let  $M_1$ and $M_2$ be two formal manifolds, and let $a_1\in M_1$, $a_2\in M_2$.
By \cite[Theorem 1.9]{CSW1}, 
there are identifications 
\be \label{eq:idpro}
\underline{M_1\times M_2}=\underline{M_1}\times \underline{M_2}
\quad\text{and}\quad
(M_1\times M_2)_{\{(a_1,a_2)\}}=M_{1,\{a_1\}}\times M_{2,\{a_2\}}
\ee
as formal manifolds.

Let $G$ be a formal Lie group. 
Using \eqref{eq:underlinecommdiag} and \eqref{eq:idpro}, we have the following examples of formal Lie subgroups.
\begin{exampled}\label{exam:subgroup} Together with the multiplication morphism  $\underline m$ and the unit morphism $\underline \varepsilon$, the reduction $\underline G$ becomes a Lie group, and it is a closed formal Lie subgroup of $G$. Similarly, together with the multiplication morphism $m_{(e,e)}$ and the unit morphism $\varepsilon_e$, the formal submanifold $G_{\{e\}}$ becomes an infinitesimal formal Lie group, and it is also a closed formal Lie subgroup of $G$. 
\end{exampled} 

 \begin{remarkd}\label{rmk:Ge=Lg}
     By  the formal Lie theory theorem (see \cite[(14.2.3)]{Ha} and \cite[Theorem 5.18]{MM} or Theorem \ref{thm:eqFPpre}), there is an identification  
     $G_{\{e\}}=\mathcal{G}_{\g}$ of infinitesimal formal Lie groups (see Example \ref{exam:Lq}). 
 \end{remarkd}

\section{Formal actions}
\subsection{Formal actions on formal manifolds}
Let $G$ be a formal Lie group and let $M$ be a formal manifold.
\begin{dfn}\label{def:formalaction}
	A (left) \textbf{formal action} of $G$ on  $M$ is
	 a morphism $\mu:G\times M \rightarrow M $ of formal manifolds such that the diagrams  
    \be \label{eq:action}\begin{CD}
    G\times G \times M @>{m\times \mathrm{id}_M}>>G\times M\\@V{\mathrm{id}_G\times \mu } VV @ VV{\mu}V\\G\times M @>{\mu}>>M
\end{CD} \ee and \be \label{eq:diamunit} \xymatrix{\mathrm{Spec}(\BC)\times M\ar[rd]_{=} \ar[r]^-{\varepsilon\times \mathrm{id}_M}&G\times M\ar[d]^\mu \\&M
} \ee are commutative. When such a formal action is given, we also say that $G$ acts formally on $M$.
\end{dfn}

We write the formal action in Definition \ref{def:formalaction} as $\mu:G\curvearrowright M$, or simply $G\curvearrowright M$ when $\mu$ is clear from the context. 


 For every formal action $\mu:G\curvearrowright M$, its reduction \[
 \underline \mu:\ \underline{G}\times \underline{M}=\underline{G\times M}\rightarrow \underline{M}\quad \text{(see \eqref{eq:idpro})}\] is a smooth action of $ \underline G$ on  $\underline M$. For every $g\in G$, the composition  \be \label{eq:defofmug}\mu_g:\ M=\mathrm{Spec}(\BC)\times M\xrightarrow{\varsigma_g\times \mathrm{id}_M} G\times M \xrightarrow \mu M\ee  
is an isomorphism of formal manifolds.
Here, for a formal manifold $N$ and $b\in N$,  
\be \label{eq:varsigma}\varsigma_b=(\overline{\varsigma_b},\varsigma_b^*):\ \Spec(\BC)\rightarrow N \ee denotes the unique morphism such that $\overline{\varsigma_b}$ maps the unique point in $\Spec(\BC)$ to $b$. 
 

The following lemma is clear by using Lemma \ref{prop:asfunctors}.
\begin{lemd}\label{lem:groupasfunctors}
      A morphism $\mu: G\times M\rightarrow M$ is a formal action if and only if 
      for every formal manifold $N$, the map \[\mu_N:\ (G\times M)(N)=G(N)\times M(N) \rightarrow M(N)\] is a group action of $G(N)$ on $M(N)$. Here $\mu_N$ is as in \eqref{eq:varphiM}. 
\end{lemd}


\begin{exampled}\label{exm:lrtranslations}
The formal Lie group $G$ acts formally on itself by left or right multiplication. As usual, 
the formal action given by right multiplication is the composition
\[G\times G\xrightarrow{\nu}G\times G\xrightarrow{\mathrm{id}_G\times i} G\times G\xrightarrow m G.\] 
\end{exampled}
Here and in the following, for any formal manifolds $M_1$ and $M_2$, let
\be\label{eq:nudef}
\nu:\ M_1\times M_2 \rightarrow M_2\times M_1
\ee
denote the switching isomorphism.

\begin{dfn} Let $\mu_1: G\curvearrowright M_1$ and $\mu_2: G\curvearrowright M_2$ be two formal actions of $G$ on formal manifolds $M_1$ and $M_2$.
	A morphism  $\varphi: M_1\rightarrow M_2$ of formal manifolds is said to be  $G$-\textbf{equivariant} if the diagram 
	 \be\label{eq:G-equidef} \begin{CD}G\times M_1 @>\mathrm{id}_G\times \varphi>> G\times M_2\\@V{\mu_1}VV @VV{\mu_2}V\\M_1 @>\varphi>> M_2\end{CD}\ee commutes.
\end{dfn}


Let $H$ be another formal Lie group. 
\begin{dfn}
	  A formal action of $G$ on $H$ as automorphisms is a formal action $\mu: G\curvearrowright H$ such that the following diagram commutes:
   \be \label{eq:H-grp} \xymatrix@R-0.5pc@C+3pc{G\times H\times H\ar[r]^-{\Delta_G\times\mathrm{id}_{H}\times \mathrm{id}_{H}}\ar[dd]_{\mathrm{id}_{G}\times m}&G\times G\times H\times H \ar[r]^{\mathrm{id}_G\times \nu \times \mathrm{id}_{H}}&G\times H\times G\times H\ar[d]^{\mu \times \mu } \\&&H\times H\ar[d]^{m}\\
   G\times H \ar[rr]^{\mu}&&H.} \ee 
When such a formal action is given, we also say that $G$ acts formally on $H$ as automorphisms.
\end{dfn}

For every formal action $\mu:G\curvearrowright H$ as automorphisms, $\underline \mu: \underline G\curvearrowright \underline H$ is a smooth action as automorphisms.

The following result follows immediately from Lemmas \ref{prop:asfunctors}  and  \ref{lem:groupasfunctors}.
\begin{lemd}\label{lem:groupsactionasauto}
    A formal action $\mu: G\curvearrowright H$ is a formal action as automorphisms if and only if for every formal manifold $N$, the group action \[\mu_N:\ G(N)\curvearrowright H(N)\] is an action as group automorphisms.  Here $\mu_N$ is as in \eqref{eq:varphiM}.
\end{lemd}

\begin{exampled}\label{exam:congudef}
   The formal Lie group $G$ acts formally on itself by inner automorphisms via
\be \label{eq:Phi}
\Phi:\ G\times G \xrightarrow{(\mathrm{id}_G\times i) \times \mathrm{id}_G} G\times G \times G
\xrightarrow{\mathrm{id}_G\times \nu} G\times G\times G\xrightarrow {m\times \mathrm{id}_G} G\times G\xrightarrow{m} G.
\ee  
\end{exampled}
\subsection{Stabilizer formal Lie subgroups}
Let $\mu :G\curvearrowright M$ be a formal action, and let $a\in M$.
Form   the composition 
 \[\mu^a:\ G=G\times \Spec(\BC)\xrightarrow{\mathrm{id}_G\times \varsigma_a} G\times M\xrightarrow{\mu} M.\] 
\begin{dfn}
A \textbf{stabilizer formal Lie subgroup} at $a$ (with respect to  the formal action $\mu: G\curvearrowright M$) is a formal Lie subgroup $H$ of $G$ such that the diagram 
\be\label{eq:diamofGa} \begin{CD}
    H @>>> \Spec(\BC) \\ @V\iota VV @VV \varsigma_a V \\ G @>\mu^a>> M
\end{CD}\ee is a Cartesian diagram in the category of formal manifolds.
\end{dfn}
By the universal property of the Cartesian diagram, a stabilizer formal Lie subgroup at \(a\), if it exists, is unique. Accordingly, whenever it exists, we denote it by \(G_a\) and refer to it as the stabilizer formal Lie subgroup at \(a\).
In such case, the reduction $\underline{G_a}$ of the stabilizer formal Lie subgroup $G_a$ is the stabilizer Lie subgroup at $a$ with respect to  the smooth action $\underline{\mu}: \underline G\curvearrowright \underline M$.



Recall from \cite[Definition 1.5]{CSW4} the notion of standardizability near a point $b\in M_1$  for a morphism \(\varphi:M_1\to M_2\) of formal manifolds.
\begin{prpd}\label{prop:stabl}

Let $\mu :G\curvearrowright M$ be a formal action, and let $a\in M$.

\noindent (a) For every $g\in G$, the morphism $\mu^a$ is standardizable near $g$.

\noindent (b) There exists a unique closed formal Lie subgroup of $G$ such that it is the stabilizer formal Lie subgroup at $a$.    
\end{prpd}

The following result is straightforward by Proposition \ref{prop:stabl}.
\begin{cord}\label{cor:kervarG}
    Let $\varphi: G_1\rightarrow G_2$ be a homomorphism between formal Lie groups. Then there exists a unique closed  formal Lie subgroup $\ker(\varphi)$ such that the diagram  
    \be \begin{CD}
    \ker(\varphi) @>>> \Spec(\BC) \\ @V\iota VV @VV \varepsilon V \\ G_1@>\varphi>> G_2
\end{CD}\ee is a Cartesian diagram in the category of formal manifolds.
    
\end{cord}
When $\varphi :G_1\rightarrow G_2$ is a homomorphism between Lie groups, $\ker(\varphi)$ is a closed Lie subgroup which is precisely the kernel of $\varphi$.  

The remainder of this subsection is devoted to proving Proposition \ref{prop:stabl}. We first fix some notation that will be used frequently throughout this paper.

Let $E_1$ and $E_2$ be two complete LCS. For  $\eta_1\in E_1'$ and $\eta_2\in E'_2$, let 
\be\label{eq:ra12} \la\eta_1, u\ra_1\qaq\la\eta_2,u\ra_2\ee denote the image of $u\in E_1\wh\otimes_{\pi} E_2$ under the continuous maps
\[E_1\wh\otimes_{\pi} E_2\xrightarrow{\eta_1\otimes \mathrm{id}_{E_2}}E_2\qaq E_1\wh\otimes_{\pi} E_2\xrightarrow{\mathrm{id}_{E_1}\otimes\eta_2 }E_1,\] respectively. 
Let $E_3$ be another complete LCS.  For every $\eta_3\in E_3'$, let
\be\label{eq:ra123} \la\eta_1\otimes\eta_2,v\ra_{1,2}\qaq\la\eta_2\otimes\eta_3,v\ra_{2,3} \ee denote 
the image of $v\in E_1\wh\otimes_{\pi} E_2\wh\otimes_{\pi} E_3$ under the continuous maps
\[E_1\wh\otimes_{\pi} E_2\wh\otimes_{\pi} E_3\xrightarrow{\eta_1\otimes\eta_2\otimes \mathrm{id}_{E_3}}E_3 \qaq E_1\wh\otimes_{\pi} E_2\wh\otimes_{\pi} E_3\xrightarrow{\mathrm{id}_{E_1}\otimes\eta_2\otimes\eta_3}E_1,\]
respectively. 
For every $\eta\in \g$, define a 
  $\BC$-linear map \[ D_\eta:\ \wh\CO_{G,e}\rightarrow \wh{\CO}_{G,e}, \quad f\mapsto \la\eta, m^*_{(e,e)}(f) \ra_2.\]
  

As usual, for a commutative unital $\BC$-algebra $A$, an element $D\in \Hom_{\BC} (A,A)$ is called a derivation if
\[D(b_1b_2)=D(b_1)b_2+b_1D(b_2)\quad \text{for all $b_1,b_2\in A$.}\] 


\begin{lemd}\label{lem:deronOe} For every $\eta\in \g$, $D_\eta$ is a derivation. 
\end{lemd}
\begin{proof}
This is easy to verify and we omit the details. 
\end{proof}
\begin{lemd}\label{lem:tauonmuf}
   Let $\tau\in\g$ with $d(\mu^a)_e(\tau)=0$. 
    Then \be\label{eq:Dtauonmuf} D_{\tau}\big((\mu^a)^*_e(f)\big)=0\ee for all $f\in \wh\CO_{M,a}$.
\end{lemd}
\begin{proof}
    By \eqref{eq:action}, 
    we have that \begin{eqnarray*} D_{\tau}\big((\mu^a)^*_e(f)\big) &=&
        \la\tau, m^*_{(e,e)}(\mu^a)^*_e(f)\ra_2\\ &=& \la\tau\otimes\delta_a, (m\times\mathrm{id}_M)^*_{(e,e,a)}\mu^*_{(e,a)}(f)\ra_{2,3}\\
        &=&\la \tau\otimes\delta_a,  (\mathrm{id}_G\times \mu)^*_{(e,e,a)}\mu^*_{(e,a)}(f)\ra_{2,3}. \end{eqnarray*} The lemma then follows by noting that 
         \[\la\tau\otimes \delta_a, \mu^*_{(e,a)}f\ra=\la\tau, \la \delta_a,\mu^*_{(e,a)}f\ra_2\ra=\la d(\mu^a)_e(\tau), f \ra=0.\]
\end{proof}

\begin{lemd}\label{lem:Gestand}
   The morphism $(\mu^a)_{e}: G_{\{e\}}\rightarrow M_{\{a\}}$ is standardizable near $e\in G_{\{e\}}$.
\end{lemd}
\begin{proof}
  Let $r$ be the rank of the differential of $(\mu^a)_{e}$ at $e$.
    Using the constant rank theorem of formal manifolds (see \cite[Theorem 3.6]{CSW4}), choose a coordinate system $(x_1,x_2,\dots, x_n)$ of $ M_{\{a\}}$ and  a coordinate system $(y_1,y_2,\dots,y_k)$ of $G_{\{e\}}$ such that the morphism $(\mu^a)_{e}$ is given by 
    \[(x_1,x_2,\dots,x_r,x_{r+1},\dots,x_n)\mapsto (y_1,y_2,\dots,y_r,h_{1},\cdots,h_{n-r}),\] 
   where \[h_i=\sum_{j=1}^{r}c_{ij}y_j+\sum_{J\in \BN^{k};|J|\geq 2}c_{iJ}y^J \quad (i=1,2,\dots,n-r)\] 
   for some $c_{ij},c_{iJ}\in \BC$. 
   Here for each $k$-tuple $J=(j_1,j_2,\dots,j_k)\in \BN^k$, we use  the following usual multi-index notations:
\[  |J|=j_1+j_2+\cdots+j_k,\qaq
  y^J=y_1^{j_1}y_2^{j_2}\cdots y_k^{j_k}.
\] By \cite[Lemma 3.4]{CSW1} and Lemma \ref{lem:deronOe}, we may (and do) further assume that 
    \be\label{eq:Dtau=y} D_{\tau_j}=\partial_{y_j}  \quad(r+1\leq j\leq k),\ee where $\{\tau_j\}_{r+1\leq j\leq k}$ is a basis of the kernel of the differential of $(\mu^a)_e$ at $e$. 

By \eqref{eq:Dtauonmuf} and \eqref{eq:Dtau=y}, it follows that $c_{i,J}=0$ whenever
$1\le i\le n-r$ and
$J\in \BN^{k}\setminus (\BN^{r}\times \{0\}^{k-r})$ with $|J|\ge 2$.  Then it is routine to verify that $(\mu^a)_e$ satisfies the criterion in \cite[Theorem 1.7]{CSW4} for a morphism to be standardizable near a point. Hence the lemma follows.
\end{proof}

\noindent\textbf{Proof of Proposition \ref{prop:stabl}:} 
The assertion (b) follows from the assertion (a) and \cite[Proposition 1.12]{CSW4}. 
Therefore, it remains to prove the assertion (a).

By \cite[Theorem 1.7]{CSW4}, a morphism $\varphi: M_1\rightarrow M_2$ of formal manifolds is standardizable near a point $b\in M_1$ if and only if there is an open neighborhood  $U$ of $b$ in  $M_1$, such that $\varphi|_{U}: U\rightarrow M_2$ has constant rank, and that  for each point $b'\in U$, the morphism 
$\varphi_{b'}: M_{1,\{b'\}}\rightarrow M_{2,\{\overline{\varphi}(b')\}}$ is standardizable. Therefore, it suffices to show that 
\be\label{eq:standardizable}
\begin{aligned}
&\mu^a \text{ has constant rank, and for all } g\in G,\\ &\text{the morphism }
(\mu^a)_{g}: G_{\{g\}}\rightarrow M_{\{\overline{\mu}((g,a))\}}
\text{ is standardizable near } g\in G_{\{g\}}.
\end{aligned}
\ee

Note that the morphism $\mu^a: G\rightarrow M $ is $G$-equivariant, where $G$ carries the formal action $m: G\curvearrowright G$ given by left multiplication. 
By Lemma \ref{prop:asfunctors}, for every $g\in G$, the diagram \be\label{eq:muaG-equi} \begin{CD}
    G@>\mu^{a}>> M \\ @Vm_{g}VV @AA\mu_{g^{-1}}A \\ G @>\mu^{a}>> M
\end{CD}\ee commutes. 
By using  \eqref{eq:muaG-equi} and Lemma \ref{lem:Gestand}, it is clear that \eqref{eq:standardizable} holds. This completes the proof. \qed


\subsection{Fixed points}




Let $\mu=(\overline{\mu},\mu^*): G\curvearrowright M$ be a formal action.

\begin{dfn}\label{def:stable}
    An immersed formal submanifold $S$ of $M$ is called $G$-\textbf{stable} (with respect to the formal action $\mu: G\curvearrowright M$) if there exists a formal action $\mu': G\curvearrowright S$ such that 
    the inclusion morphism $\iota: S\rightarrow M$ is $G$-equivariant. 
\end{dfn}
It follows from \eqref{inj00} that the formal action $\mu'$ in Definition \ref{def:stable} 
is unique.
\begin{dfn}
\noindent(a)  A point $a\in M$ is said to be $G$-\textbf{fixed} if the morphism \[\varsigma_a:\ \Spec(\BC)\rightarrow M \quad \text{(see \eqref{eq:varsigma})}\] is $G$-equivariant, where $ \Spec(\BC)$ carries the trivial action of $G$.

\noindent(b) A point $a\in M$ is said to be $\underline G$-\textbf{fixed} if the closed formal submanifold $M_{\{a\}}$ of $M$ is $G$-stable (see Example \ref{exam:subgroup}).
\end{dfn}

\begin{prpd}\label{prop:GonMa}
\noindent (a) A point $a\in M$ is $G$-fixed if and only if the stabilizer formal Lie subgroup at $a$  is $G$.

\noindent (b) A point $a\in M$ is $\underline G$-fixed if and only if $\overline\mu\big((g,a)\big)=a$ for all $g\in G$.
\end{prpd}
\begin{proof}
     The assertion (a) is clear by \eqref{eq:diamofGa} (replacing $H$ by $G$), while
    the assertion (b) follows from  \cite[Corollary 4.21]{CSW4} and Lemma \ref{lem:groupasfunctors}.
\end{proof}
By Proposition \ref{prop:GonMa}, every $G$-fixed point is $\underline G$-fixed. The following result follows immediately from Lemma \ref{lem:groupasfunctors} and \eqref{eq:underlinecommdiag}.
\begin{prpd}
 Let $\varphi=(\overline{\varphi},\varphi^*): M\rightarrow M'$ be a $G$-equivariant morphism, where $M'$ is another formal manifold carrying a formal action $ G\curvearrowright M'$.  Let $a\in M$ be a $ \underline G$-fixed point. Then $\overline \varphi(a)\in M'$ is also $\underline G$-fixed, and the morphism \[\varphi_a:\  M_{\{a\}}\rightarrow M'_{\{\overline\varphi(a)\}}\] is also 
    $G$-equivariant.
\end{prpd}

\section{Isotropy representations}
\subsection{Isotropy representations}
Let $\mu:G\curvearrowright M $ be a formal action.

\begin{prpd}\label{prop:gonTwithfixed}
    Let 
    $a\in M$ be a $G$-fixed point. 

\noindent(a)  The map \be\label{eq:dmuonT} d\mu:\ \g\times \RT_a(M)\rightarrow \RT_a(M),\quad (\tau,\eta)\mapsto {}^t\mu^*_{(e,a)}(\tau\otimes \eta) \ee is well-defined, and it is a Lie algebra action.  Here ${}^t\mu^*_{(e,a)}(\tau\otimes\eta)$ is the image of $\tau\otimes\eta$ under the map 
\[\mathrm{Dist}_{e}(G)\otimes_{\Ri}\mathrm{Dist}_a(M)=\mathrm{Dist}_{(e,a)}(G\times M)\xrightarrow{{}^t\mu^*_{(e,a)}}\mathrm{Dist}_a(M) \quad \text{(see Lemma \ref{lem:T3})}.\]

\noindent(b) Let $\varphi=(\overline{\varphi},\varphi^*): M\rightarrow M'$ be a $G$-equivariant morphism, where $M'$ is another formal manifold carrying a formal action $ G\curvearrowright M'$. 
Then $\overline \varphi(a)\in M'$ is also $ G$-fixed, and the differential 
\[d\varphi_a: \ \RT_a(M)\rightarrow \RT_{\overline{\varphi}(a)}(M')\] is 
    $\g$-equivariant.
\end{prpd}
\begin{proof}
   Since the point $a\in M$ is $G$-fixed, we have that 
   \[ {}^t\mu^*_{(e,a)}(\tau\otimes\delta_a)=0\]
for each $\tau\in \g$.    Using this, it is easy to verify that for every $\eta\in \RT_a(M)$, 
    \begin{eqnarray*}
       && \Delta ({}^t\mu^*_{(e,a)}(\tau\otimes\eta)) ={}^t\mu^*_{(e,a)}(\tau\otimes\eta)\otimes\delta_a+\delta_a\otimes{}^t\mu^*_{(e,a)}(\tau\otimes\eta),
    \end{eqnarray*} where $\Delta$ denotes the comultiplication of  
    the $\wh\otimes_\Ri$-coalgebra $\mathrm{Dist}_a(M)$. 
      Thus
    the map \eqref{eq:dmuonT} is well-defined.
 Moreover,  by using 
the commutative diagram \eqref{eq:action}, we obtain  that\be \label{eq:byHmodule}{}^t\mu^*_{(e,a)}((\tau_1\tau _2)\otimes \eta)={}^t\mu^*_{(e,a)}(\tau_1\otimes {}^t\mu^*_{(e,a)}(\tau_2\otimes \eta))\ee  for all  $\tau_1, \tau_2\in\g$ and $\eta\in\RT_a(M)$.  Then the map \eqref{eq:dmuonT} is a Lie algebra action.
 This proves the assertion (a), and the assertion (b) is straightforward. 
\end{proof}

Write $\underline \g$ for the complexified Lie algebra of $\underline G$. For a Lie group, let ``$\exp$" denote the exponential map on its real Lie algebra.





\begin{prpd}\label{prop:GonT}Let 
$a\in M$ be a $\underline G$-fixed point.

\noindent(a) 
    The map \be\label{eq:smoothGontangent} \underline G\times \mathrm{T}_a(M)\rightarrow \mathrm{T}_a(M),\quad (g,\eta)\mapsto d(\mu_{g})_{a}(\eta)\ee  is a smooth representation of the Lie group $\underline G$.
    Furthermore, the differential  of \eqref{eq:smoothGontangent} is given by 
    \be \label{eq:undergonT}\underline \g\curvearrowright \mathrm{T}_a(M), \quad (\tau, \eta)\mapsto {}^t\mu^*_{(e,a)}(\tau\otimes\eta).\ee 
    
 \noindent(b) Let $\varphi=(\overline{\varphi},\varphi^*): M\rightarrow M'$ be a $G$-equivariant morphism, where $M'$ is another formal manifold carrying a formal action $ \mu': G\curvearrowright M'$.  Then the differential 
 \[d\varphi_{a}: \ \RT_a(M)\rightarrow \RT_{\overline{\varphi}(a)}(M')\] is $\underline G$-equivariant.
\end{prpd}
\begin{proof} Let $\eta\in \mathrm{T}_a(M)$ and  $f\in \CO_M(M)$.
Note that for any two formal manifolds $M_1$ and $M_2$,  \be\label{eq:OM1M2} \CO_{M_1\times M_2}(M_1\times M_2)=\CO_{M_1}(M_1)\wh\otimes_{\pi} \CO_{M_2}(M_2) \ee
as formal algebras (see \cite[Theorem 1.9]{CSW1}). 
   Then we have that 
\be\label{eq:Evaboutfixedpoint} \la d(\mu_{g})_{a}(\eta), f\ra=\la \eta, \mu^*_g(f)\ra=\la\Ev_g\otimes\eta,\mu^*(f)\ra=\la\Ev_g, \la\eta, \mu^*(f)\ra_2\ra \ee for each $g\in G$, 
where $\Ev_g$ denotes the composition 
$\CO_G(G)\rightarrow \wh\CO_g\xrightarrow{\delta_g}\BC $.
This, together with the fact that $\la\eta, \mu^*(f)\ra_2\in \CO_G(G)$,  implies that the map
 \[ G\rightarrow \BC,\quad g\mapsto \la d(\mu_{g})_{a}(\eta),f\ra\]
 is continuous, and hence the map  \[ G\rightarrow \mathrm{T}_a(M), \quad g\mapsto d(\mu_{g})_{a}(\eta)\] 
 is continuous as well. 
 Thus,  the map \eqref{eq:smoothGontangent} is a smooth representation of $\underline G$.

 For every  $\tau\in \Lie(\underline G)$,
it follows from \eqref{eq:Evaboutfixedpoint} that
\begin{eqnarray*}
   &&\lim_{t\rightarrow0 }\la\frac{d(\mu_{\exp(t\tau)})_a(\eta)-d(\mu_{e})_a(\eta)}{t},f\ra\\ &=& \lim_{t\rightarrow0 }
   \la \frac{\Ev_{\exp(t\tau)}-\Ev_{e}}{t} , \la\eta, \mu^*f\ra_2\ra \\&=&\la\tau,\la\eta,\mu^*(f) \ra_2 \ra=\la\tau\otimes\eta,\mu^*(f) \ra=\la {}^t\mu^*_{(e,a)}(\tau\otimes\eta), f \ra.
\end{eqnarray*} 
This implies that the differential  of \eqref{eq:smoothGontangent} is given by 
    \eqref{eq:undergonT}. 
Finally, the assertion (b) is straightforward by \eqref{eq:chainruleofdvarphi} and the assertion (a). 
\end{proof}

In Section \ref{Sec:adjoint}, we will prove that $(\g,\underline G)$ is a Lie pair.

\begin{dfn} Let $\mu:G\curvearrowright M $ be a formal action, and let $a\in M$ be a $G$-fixed point. The pair of actions \eqref{eq:dmuonT} and \eqref{eq:smoothGontangent} on $\RT_a(M)$ induced by $\mu$ is called the \textbf{isotropy representation} of $(\g,\underline G)$ at $a\in M$.
\end{dfn}

\subsection{Adjoint representations}
Let $\mu: G\curvearrowright H$ be a formal action of $G$ on a formal Lie group $H$ as automorphisms. 
By Lemma \ref{prop:asfunctors}, the point $e\in H$ is $ G$-fixed.
In view of Proposition \ref{prop:gonTwithfixed}, the formal action $\mu$ yields a Lie algebra action  \be\label{eq:dmuonh} d\mu: \g\curvearrowright \h,\quad (\tau,\eta)\mapsto {}^t\mu^*_{(e,e)}(\tau\otimes \eta). \ee 
The following result is straightforward. 
\begin{lemd}\label{prop:gonhwithfixed}
    The action \eqref{eq:dmuonh}
     acts as derivations.
\end{lemd}

For every $g\in G$, the morphism \be\label{eq:muh}  \mu_g:\ H=\mathrm{Spec}(\BC)\times H \xrightarrow{\varsigma_g\times\mathrm{id}_H}G\times H\xrightarrow{\mu} H \ee  is a formal Lie group automorphism.  By Proposition \ref{prop:Liealg}, the differential 
\[d(\mu_{g})_e:\ \h \rightarrow \h\] of $\mu_g$ is a homomorphism of Lie algebras.

The following result follows from Proposition \ref{prop:GonT} and Lemma \ref{prop:gonhwithfixed}.
\begin{prpd}\label{prop:GonTH}Let $\mu:G\curvearrowright H$ be a formal action as automorphisms. Then the map \be\label{eq:smoothGonliealg} \underline G\times \h \rightarrow \h,\quad (g,\eta)\mapsto d(\mu_{g})_e(\eta)\ee  is a smooth representation of the Lie group $\underline G$ as Lie algebra automorphisms.
    Furthermore, the differential \be\label{eq:LieGactsonh} \underline\g\curvearrowright \h, \quad (\tau, \eta)\mapsto {}^t\mu^*_{(e,e)}(\tau\otimes\eta)\ee of \eqref{eq:smoothGonliealg} is a Lie algebra action  as  derivations.    
\end{prpd}

Recall from Example \ref{exam:congudef} that $G$ acts formally on itself via the morphism
\be \label{eq:congudef}
\Phi:\ G\times G \xrightarrow{(\mathrm{id}_G\times i) \times \mathrm{id}_G} G\times G \times G
\xrightarrow{\mathrm{id}_G\times \nu} G\times G\times G\xrightarrow {m\times \mathrm{id}_G} G\times G\xrightarrow{m} G.
\ee
By Proposition \ref{prop:GonTH} and Lemma \ref{prop:gonhwithfixed}, the formal action \eqref{eq:congudef} induces a Lie group representation
\be\label{eq:Addef}
\mathrm{Ad}:\ \underline G\curvearrowright \g,\quad (g,\eta)\mapsto d(\Phi_g)_e(\eta)
\ee of $\underline G$ as Lie algebra automorphisms, as well as a Lie algebra action  \be\label{eq:dPhi} d\Phi: \g\curvearrowright \g,\quad (\tau,\eta)\mapsto {}^t\Phi^*_{(e,e)}(\tau\otimes \eta) \ee as derivations.
 Using Lemmas \ref{lem:difftimes} and \ref{lem:inverse}, the following result is easy to verify.
\begin{lemd}\label{prop:[g,g]}
The Lie algebra action \eqref{eq:dPhi} equals the adjoint action 
\be\label{eq:[g,g]} \ad:\  \g\curvearrowright \g, \quad (\tau,\eta)\mapsto [\tau,\eta].\ee
\end{lemd}

\begin{dfn} The pair $(\ad,\Ad)$ (in \eqref{eq:[g,g]} and \eqref{eq:Addef}) is called the \textbf{adjoint representation} of $(\g,\underline G)$.
\end{dfn}

When $G$ is a Lie group, the formal action \eqref{eq:congudef} is precisely the smooth action of $G$ on itself by inner automorphisms. Moreover, in this case, \be \label{eq:AdofLiegroup}\text{\eqref{eq:Addef} is the usual adjoint representation of $G$ on $\g$  (see \cite[3.46]{War})}.\ee  
\subsection { Lie pairs associated to formal Lie groups}
\label{Sec:adjoint}
   Let $G$ be a formal Lie group.
   By \cite[Lemma 2.7]{CSW4} and Proposition \ref{prop:Liealg}, 
the differential $d\iota_e: \underline \g\rightarrow \g$ of the inclusion morphism $\iota: \underline G\rightarrow G$ is an injective Lie algebra homomorphism. 
\begin{prpd}\label{prop:GtoLiepair} Together with the adjoint representation  $\Ad: \underline G\curvearrowright \g $ and the  injective homomorphism  $d\iota_e: \underline\g \rightarrow \g$, the pair $( \g, \underline G)$  is a Lie pair. 
Furthermore, the assignment
\[G\longmapsto (\g,\underline G) \qaq (\varphi:G_1\rightarrow G_2)\longmapsto((d\varphi_e,\underline \varphi ): (\g_1,\underline {G_1})\rightarrow (\g_2,\underline{G_2}))\] is a functor from the category of formal Lie groups to the category of Lie pairs.
\end{prpd}
\begin{proof}
 The second assertion is straightforward. For the first one, by Proposition \ref{prop:GonTH} and Lemma \ref{prop:[g,g]}, the differential $ \underline \g\curvearrowright \g$ of $\Ad: \underline  G\curvearrowright \g$ equals the action
		\[
		\underline \g \curvearrowright \g,\quad (\tau,\eta)\mapsto [\tau,\eta].
		\]
  It suffices to prove that \be \label{eq:homoG-equi} \text{the homomorphism $d\iota_e: \underline \g\rightarrow \g $ is $\underline G$-equivariant,}\ee where $\underline \g$ carries the usual adjoint representation of $\underline G$, and $\g$ carries the representation of $\underline G$ as in \eqref{eq:Addef}.   

  Using \eqref{eq:underlinecommdiag}, the diagram \[\begin{CD} \underline G \times \underline G @>\underline \Phi>> \underline G \\ @V\iota\times \iota VV @VV \iota V \\ G\times G @>\Phi>> G\end{CD}\] is commutative. This implies that the morphism $\underline G\rightarrow G$ is $\underline G$-equivariant, where $\underline G$ acts on $\underline G$ and $G$ by inner automorphisms. Then \eqref{eq:homoG-equi} follows by using Proposition \ref{prop:GonT} (b). This finishes the proof.
\end{proof}


\section{Semidirect products  and normal formal Lie subgroups}
\subsection{Semidirect products of formal Lie groups}\label{sec:semi}
Let $\mu: G\curvearrowright H$ be a formal action as automorphisms. 
\begin{lemd}\label{prop:semiLiegroups}
    Together with the multiplication morphism 
\be\label{eq:mGltimesH}\xymatrix@R-0.8pc@C+2.9pc{m:\ G\times H \times G\times H \ar[r]^{\mathrm{id}_G\times \nu \times \mathrm{id}_H}&G\times G \times H \times H \ar[d]_{\mathrm{id}_G\times(\mathrm{id}_G\times i)\times \mathrm{id}_H \times \mathrm{id}_H}\\
&G\times G\times G\times H \times H\ar[r]^-{\mathrm{id}_G\times \mathrm{id}_G\times\mu \times \mathrm{id}_H}&G\times G\times H\times H \ar[d]_{m\times m}\\
&&G\times H, }\ee 
 and the unit morphism $\varepsilon: \Spec(\BC)=\Spec (\BC)\times \Spec(\BC)\xrightarrow{\varepsilon\times \varepsilon} G\times H$, the formal manifold $G\times H$ becomes a formal Lie group.
Moreover, the inversion morphism of this formal Lie group is \be \label{eq:iGH} i:\ 
	G\times H \xrightarrow {(i\times\mathrm{id}_G)\times i }G\times G \times H \xrightarrow{\mathrm{id}_G\times \mu} G\times H.\ee
     Here,  $\nu$ is the switching isomorphism (see \eqref{eq:nudef}). 
\end{lemd}
\begin{proof}
 It follows from Lemma \ref{lem:groupsactionasauto} that for every formal manifold $N$, the group action $\mu_N:G(N)\curvearrowright H(N)$ acts as group automorphisms. Recall that the corresponding semidirect
product $G(N)\ltimes H(N)$ is a group whose underlying set is
\[
   G(N)\times H(N)=(G\times H)(N).
\]
The lemma then follows from Lemma \ref{prop:asfunctors}.
\end{proof}

\begin{dfn}
    The formal Lie group \((G\times H,m,\varepsilon)\) constructed in Lemma  \ref{prop:semiLiegroups} is called the \textbf{semidirect product} of \(G\) and \(H\) with respect to \(\mu\). 
\end{dfn}

As usual, the semidirect product of $G$ and $H$ with respect to $\mu$ is denoted by \(G\ltimes_{\mu} H\), or simply by \(G\ltimes H\) when \(\mu\) is clear from the context.
Note that \be\label{eq:underlineofsemi}\underline{G\ltimes_{\mu} H}=\underline G\ltimes_{\underline{\mu}} \underline H\ee
as Lie groups.

By using Lemma \ref{prop:asfunctors}, the injective immersions 
 \be \label{eq: imsemi}G=G\times \Spec (\BC)\xrightarrow{\mathrm{id}_G\times \varepsilon} G\ltimes H \quad\text{and}\quad  H=\Spec (\BC) \times H \xrightarrow{\varepsilon\times \mathrm{id}_H} G\ltimes H \ee are homomorphisms of formal Lie groups. Recall from Lemma \ref{lem:T3} that  \[\RT_{(e,e)}(G\times H)=(\g\otimes\delta_e)\oplus(\delta_e\otimes\h)=\g\oplus\h\]
 as vector spaces. 
 The following result is obvious. 
\begin{lemd}\label{lem:gtogh}
    The differentials
    \[\g\rightarrow \Lie_{\BC}(G\ltimes H), \quad \tau\mapsto \tau\otimes\delta_e\qaq\h\rightarrow \Lie_{\BC}(G\ltimes H),\quad \eta \mapsto \delta_e\otimes\eta\]  of \eqref{eq: imsemi} are both injective Lie algebra homomorphisms.
\end{lemd}
 

 Let $\g\ltimes_{d\mu}\h$ denote the semidirect product of the Lie algebras $\g$ and $\h$ with respect to the Lie algebra action     \be\label{eq:dmuonh2} d\mu: \g\curvearrowright \h,\quad (\tau,\eta)\mapsto {}^t\mu^*_{(e,e)}(\tau\otimes \eta) \quad \text{(see \eqref{eq:dmuonh})}.\ee
  The following result is a generalization of \cite[Proposition 1.124]{Kn}.

\begin{prpd}\label{prop:Liealgofsemi} Let $\mu:G\curvearrowright H$ be a formal action as automorphisms.
Then, under the identification $\RT_{(e,e)}(G\times H)=\g\oplus\h$, we have that
    \[\Lie_{\BC}(G\ltimes H)=\g\ltimes_{d\mu}\h\] as Lie algebras.
\end{prpd}
\begin{proof} Let $\tau\in \g $ and $\eta \in \h$. 
Note that $d\mu_{(e,e)}( \delta_e\otimes \eta) = \eta$. 
It then follows from 
  \eqref{eq:mGltimesH} and Proposition \ref{prop:DistasHopf} that in the Lie algebra $ \Lie_{\BC}(G\ltimes H)$, 
\[[\tau\otimes \delta_e,\delta_e\otimes \eta ]=\delta_e\otimes {}^t\mu^*_{(e,e)} (\tau\otimes \eta).\] 
This, together with 
 \cite[Proposition 1.22]{Kn} and Lemma \ref{lem:gtogh}, implies the proposition. 
\end{proof}

 \subsection{Normal formal Lie subgroups}\label{sec:mormal}  Let $G$ be a formal Lie group.
   \begin{dfn}\label{def:normal}
   	A formal Lie subgroup
   	$H$ of $G$ is said to be  \textbf{normal} if 
    $H$ is $G$-stable with respect to  the formal action \eqref{eq:Phi} of $G$ on itself by inner automorphisms (see Definition \ref{def:stable}).
   \end{dfn}	



For every normal formal Lie subgroup $H$ of $G$, the reduction 
$\underline H$ is a normal Lie subgroup of $\underline G$.
Using Lemma \ref{prop:asfunctors}, the following result is clear.  
\begin{lemd}\label{lem:normalset}
  A formal Lie subgroup  $H$ of $G$ is normal if and only if for every formal manifold $N$, 
  the subgroup $\iota_N(H(N))$ of $G(N)$
  is normal. Here $\iota: H\rightarrow G$ is the inclusion morphism, and $\iota_N$ is as in \eqref{eq:varphiM}.
\end{lemd}

 We have the following result by using Lemma \ref{prop:[g,g]}.
 
 \begin{prpd}  Let $H$ be a normal formal Lie subgroup of $G$. 
     Then its complex   Lie algebra $\h$ is an ideal of $\g$.
 \end{prpd}

The following are some examples of normal formal Lie subgroups.

\begin{exampled}\label{exam:Genormal}
\noindent(a) By Proposition \ref{prop:GonMa} (b), the closed formal Lie subgroup $G_{\{e\}}$ of $G$ is normal. 

\noindent(b) 
For every homomorphism  $\varphi:G_1\rightarrow G_2$ between formal Lie groups, the closed formal Lie subgroup $\ker(\varphi)$ of $G_1$ is normal (see Corollary \ref{cor:kervarG}).

\noindent(c) 
For every formal action $G_1\curvearrowright G_2$ as automorphisms, the formal Lie subgroup  of $G_1\ltimes G_2$ determined by the injective immersion \[G_2=\Spec(\BC)\times G_2\xrightarrow{\varepsilon\times \mathrm{id}_{G_2}} G_1\ltimes G_2\] as in \eqref{eq:StoG} is normal.  
\end{exampled}




\section{Quotients}
\subsection{Quotients with respect to  formal actions} 
	

Let $\mu=(\overline{\mu},\mu^*): G\curvearrowright M $ 
be a (left) formal action. 
\begin{dfn}\label{def:G-inv} 
 A morphism $\varphi: M\rightarrow M'$ of formal manifolds is said to be  $G$-invariant (with respect to $\mu$) if  the diagram  \be\label{eq:G-invariant} \begin{CD}G\times M @>\mu >>M\\@V\mathrm{pr}_2VV @VV\varphi V\\M@>\varphi>> M' \end{CD} \quad \quad\text{($\mathrm{pr}_2$ is the projection morphism)} \ee commutes. 
\end{dfn} 


\begin{dfn}\label{def:quo} 
	 A  \textbf{quotient} of $M$ by $G$ (with respect to $\mu$) is a  pair $(M_0,\pi)$ consisting of a formal manifold  $M_0$ and a $G$-invariant morphism $\pi:M\rightarrow M_0$, satisfying the following universal property: for each pair $(M', \varphi)$ consisting of a formal manifold  $M'$ and a $G$-invariant morphism $\varphi:M\rightarrow M' $, there exists a unique morphism $\varphi': M_0\rightarrow M'$ such that
    the diagram 
\[\xymatrix{ M \ar[rd]^{\varphi}\ar[d]_{\pi}\\M_0\ar[r]^{\varphi'} &M' }\]
    commutes. 
\end{dfn}
It is clear that a quotient of $M$ by $G$, if it exists, is unique. Accordingly, whenever it exists, we denote it by $(G\backslash M,\pi)$ and refer to it as the quotient of $M$ by $G$. 
In this case, $\pi$ is called the quotient morphism. 

\begin{remarkd} As in Definition \ref{def:formalaction}, one also has the notion of a right formal action, to be denoted by $M\curvearrowleft G$.  For a right formal action $\mu': M\curvearrowleft G$, a $G$-invariant morphism $\varphi: M\rightarrow M'$ is similarly defined as in Definition \ref{def:G-inv}. 
As in Definition \ref{def:quo}, one also has the notion of the quotient $(M/G,\pi)$ with respect to $\mu'$.
All definitions and results in this subsection have analogous statements for right formal actions, with the evident modifications.
\end{remarkd}



\begin{dfn}
Let $U$ be an open subset of $M$ satisfying
$\overline{\mu}(G\times U)\subset U$. 
A  formal function  $f\in \CO_M(U) $ is said to be $G$-invariant if the image of $f$ under the homomorphism  
\[\mu^*_{U}:\ \CO_M(U)\rightarrow\CO_{G\times M}(G\times U)=\CO_{G}(G)\wh\otimes_{\pi}\CO_{M}(U)\quad \text{(see \eqref{eq:OM1M2})}\]
is $1\otimes f $. 
We denote by $\CO_M(U)^{G}$  the subspace of $\CO_M(U)$ consisting of all $G$-invariant formal functions.
\end{dfn}

For every $G$-invariant morphism $\varphi: M\rightarrow M'$ of formal manifolds, it follows from \eqref{eq:G-invariant} that  the image of the homomorphism 
    \[\varphi^*_{U'}: \ \CO_{M'}(U')\rightarrow \CO_M(\overline \varphi^{-1}(U')) \quad (\text{ $U'$ is an open subset  of $M'$ })
    \]
     is contained in $\CO_M(\overline \varphi^{-1}(U'))^{G}$.



\begin{lemd}\label{lem:quoofsimple}
     Let $M_0$ be a formal manifold.  
     
    \noindent(a) The quotient of $ G\times M_0$ by $G$ with respect to the left formal action
     \be\label{eq:GactsonGtimesM} m\times \mathrm{id}_{M_0}: \ G\times G\times M_0\rightarrow G\times M_0\ee
      is $(M_0, \mathrm{pr}_2:G\times M_0\rightarrow M_0)$.  

    \noindent (b) For each open subset $V$ of $M_0$,  the homomorphism \[\mathrm{pr}^*_{2,V}: \ \CO_{M_0}(V) \rightarrow \CO_{G\times M_0}(G\times V)=\CO_G(G)\wh\otimes_{\pi}\CO_{M_0}(V), \quad f\mapsto 1\otimes f  \] is a closed embedding with  image  $\CO_{G\times M_0}(G\times V)^{G}$. 
\end{lemd}

\begin{proof} The assertion (a) follows from Lemma \ref{prop:asfunctors}. By \cite[Proposition 43.7]{T}, we have that $\mathrm{pr}^*_{2,V}$ is a closed embedding. 
For every $f_0\in \CO_{G\times M_0}(G\times V)^G$, it is straightforward that  
\[1\otimes\la \Ev_e, f_0\ra_1=f_0.\] Then $\mathrm{pr}^*_{2,V}( \la \Ev _e, f_0\ra_1)=f_0$. This finishes the proof. \end{proof}
\subsection{Quotients of formal Lie groups}
Let $G$ be a formal Lie group. 
For every formal Lie subgroup $H$ of $G$, let \be\label{eq:G/H}(G/H,\CO_{G/H})\ee denote the ringed space over $\Spec(\BC)$ defined as follows: the underlying topological space $G/H=\{gH:g\in G\}$ is endowed with the quotient topology, and the structure sheaf $\CO_{G/H}$ is given by 
\[\CO_{G/H}:\ V \mapsto \CO_{G/H}(V):=\CO_G(\overline\pi^{-1}(V))^H\quad \textrm {($V$ is an open subset of $G/H$)}, \]
where $\overline{\pi}:G\rightarrow G/H$ is the quotient map.
We write \[\pi=(\overline\pi, \pi^*):\ (G,\CO_G)\rightarrow (G/H,\CO_{G/H})\] for the morphism between ringed spaces over $\Spec(\BC)$ such that   \[\pi^*_V:\ \CO_{G/H}(V)=\CO_G(\overline\pi^{-1}(V))^H\rightarrow \CO_G(\overline\pi^{-1}(V))\] is the inclusion map. 
By abuse of notation, we will not distinguish the ringed space \eqref{eq:G/H} 
from its underlying topological space $G/H$.

Recall that $\underline \h$ and $\underline \g$ are the complexified Lie algebras of $\underline H$ and $\underline G$, respectively. 
It follows from \eqref{eq:underlinecommdiag} that $\underline \h$ is a subset of $\underline \g$. This induces a $\BC$-linear map \be\label{eq:diotaf}  \h/\underline{\h} \rightarrow \g/\underline{\g}.\ee 
\begin{dfn}
    Let $G$ be a formal Lie group. A formal Lie subgroup  $H$  of $G$ is said to be \textbf{regular} if the $\BC$-linear map \eqref{eq:diotaf} is injective.
\end{dfn}
\begin{remarkd} Recall from \cite[Definition 3.17]{CSW4}  the definition of a regular morphism between formal manifolds. A formal Lie subgroup  $H$  of $G$  is regular if and only if the inclusion morphism $\iota: H\rightarrow G$ is regular.  When $G$ is a Lie group,  all Lie subgroups are regular.
\end{remarkd}
%


Given a morphism $\varphi=(\overline{\varphi},\varphi^*): (M',\CO_{M'})\rightarrow (M,\CO_M)$ of formal manifolds, recall from \cite[Section 3.4]{CSW4} that a local section of $\varphi$ is a morphism \[\psi=(\overline{\psi},\psi^*):\ (V,\CO_M|_V)\rightarrow (M',\CO_{M'})\quad (V\ \text{is an open subset of}\ M)\] of formal manifolds such that $\varphi\circ \psi: (V, \CO_M|_V)\rightarrow (M,\CO_M)$ is the inclusion morphism. 

We will prove the following theorem in the next subsection. 
\begin{thmd}\label{thm:quo}
Let $H$ be a closed regular formal Lie subgroup of $G$. 

\noindent(a) The ringed space $G/H$ over $\Spec(\BC)$  is a formal manifold. Together with the morphism  $\pi: G\rightarrow G/H$, it is the quotient of $G$ by $H$ with respect to the right formal action
\be\label{eq:monGH} m|_{G\times H}:\ G\times H \xrightarrow{\mathrm{id}_{G}\times\iota} G\times G\xrightarrow{m} G. \ee 

\noindent (b) 
For each open subset $V$ of $G/H$, the homomorphism \[ \pi^*_{V}: \  \CO_{G/H}(V) \rightarrow \CO_G(\overline{\pi}^{-1}(V))
\] is a closed embedding, where $\CO_{G/H}(V)$ and $\CO_G(\overline{\pi}^{-1}(V))$ are equipped with the smooth topologies (see \cite[Definition 4.1]{CSW1}). 

\noindent (c) For every $gH\in G/H$, there exists an open neighborhood $V$ of $gH$ in $G/H$ and a local section 
\[\pi': \ (V,\CO_{G/H}|_{V})\rightarrow (G,\CO_G)\] of $\pi: (G, \CO_G)\rightarrow (G/H,\CO_{G/H}) $. 
 \end{thmd}


As usual, the quotient of $G$ by $H$ with respect to  \eqref{eq:monGH} is called the right quotient of $G$ by $H$.
For every closed regular formal Lie subgroup $H$ of $G$, 
\be\label{eq:underlineG/H} \text{$(\underline{G/H},\underline \pi)$ is precisely the right quotient of $\underline G$ by $\underline H$} \ee (in the category of smooth manifolds) as in \cite[Theorem 3.58]{War}. 

 \begin{remarkd} For a formal Lie subgroup $H$ of $G$, we may also define the left quotient of $G$ by $H$. When $H$ is closed and regular, the assertions in Theorem \ref{thm:quo} hold for the left quotient  $(H\backslash G,\pi)$ with the evident modifications. Furthermore, 
 there exists a unique morphism 
$i_0: G/H\rightarrow H\backslash G$ 
 such that the diagram 
\be\begin{CD}
    G @>i>> G \\ @V \pi VV @VV \pi V \\ G/H @> i_0>>H\backslash G
\end{CD} \ee
 commutes. Moreover,  the morphism $i_0$ is an isomorphism of formal manifolds. 
\end{remarkd}

\subsection{Proof of Theorem \ref{thm:quo}}

Let $H$ be a closed regular formal Lie subgroup of $G$. 
We start with  three elementary results in the theory of Lie groups (see \cite[Theorem 9.3.7 (iii) and the proof of Theorem 10.1.10]{JN}).  

\begin{lemd}
     The quotient map $\overline{\pi}: G\rightarrow G/H$  is an open map.
\end{lemd}
Take a $\BR$-linear subspace $T$ of $\Lie(\underline G)$ such that 
\[T\oplus \Lie(\underline H)=\Lie(\underline G)\] as $\BR$-linear spaces. 
\begin{lemd}\label{lem:begininsm}
There is a neighborhood $W$ of $0$ in $T$ such that the $\underline H$-equivariant smooth map \be\label{eq:VtimesH} W\times \underline H \rightarrow \underline G, \quad (\eta,h)\mapsto \exp(\eta)h\ee  is a diffeomorphism onto an open subset of $\underline G$, where $W\times \underline H$ carries the right action of $\underline H$ as in \eqref{eq:GactsonGtimesM}, and $\underline G$ carries the right action of $\underline H$  by right multiplication. 
\end{lemd}

In what follows, let $W$ be a neighborhood of $0$ in $T$  as in Lemma \ref{lem:begininsm}.
\begin{lemd}
    The map 
\[W\rightarrow G/H,\quad \eta\mapsto  \overline{\pi}(\exp(\eta))\]
is an open embedding of topological spaces. 
\end{lemd}

Since $H$ is regular, we can choose  a $\BC$-linear  subspace $E$ of $\g$ such that \be \label{eq:oplus}(\BC\otimes_{\BR} T)\oplus E\oplus\h = \g.\ee 
Recall from Lemma \ref{prop:subLs} that 
there is an immersed formal submanifold $S_E$ of $G_{\{e\}}$
with underlying set $\{e\}$ such that $E$ equals  the image of the differential $d\iota_e: \mathrm{T}_e(S_E)\rightarrow \g$.
Then there is an $H$-equivariant morphism \be\label{eq:vartheta}\phi=(\overline{\phi},\phi^*):\  W \times S_E\times H\xrightarrow{\exp|_{W}\times\mathrm{id}_{S_E}\times \mathrm{id}_H} \underline G\times  S_E\times H\xrightarrow{\iota\times\iota\times \iota} G\times G\times G\xrightarrow m G,\ee of formal manifolds, where $ W\times S_E\times H$ carries the right formal action of $H$ as in \eqref{eq:GactsonGtimesM}, and $G$ carries  the right formal action of $H$  by right multiplication. By Lemma \ref{lem:begininsm}, $\overline \phi$ is an open embedding map. 

\begin{lemd}\label{lem:loccld}
 There is an open neighborhood  $W'$ of $0$ in $W$ such that the  $H$-equivariant morphism \be\label{eq:varthetaU} \phi|_{W'\times S_E \times H}:\ W' \times S_E\times H\rightarrow (U',\CO_G|_{U'})\quad (U':=\overline{\phi}(W'\times S_E \times H))\ee
is an isomorphism of formal manifolds. 
\end{lemd} 
\begin{proof}


The differential map $d\underline\phi_{(0,e,e)}$  is bijective by  Lemma \ref{lem:begininsm}. On the other hand, it is clear that $d\phi_{(0,e,e)}$ is 
 bijective by using \eqref{eq:oplus} and Lemma \ref{lem:inverse}. The lemma then follows from the inverse function theorem of formal manifolds (see \cite[Theorem 1.3]{CSW4}). 
\end{proof}
 For each $g\in G$,  there is an   $H$-equivariant isomorphism \[l_g: \ G=\Spec(\BC)\times G\xrightarrow{\varsigma_g\times \mathrm{id}_G} G\times G\xrightarrow{m} G,\]  where $G$ carries the right formal action of $H$ given by right multiplication, and  
 $\varsigma_g$ is as in \eqref{eq:varsigma}. 
The following lemma follows from Lemmas \ref{lem:loccld} and \ref{lem:quoofsimple}.
 \begin{lemd}\label{lem:chartofquo}Let $g\in G$ and let $W', U'$ be as in Lemma \ref{lem:loccld}. 
 Then the composition of 
 \[W'\times S_E\times H\xrightarrow{\phi|_{W'\times S_E \times
H}}(U',\CO_G|_{U'})\xrightarrow{ l_g|_{U'}} (gU',\CO_G|_{gU'})\] is an $H$-equivariant isomorphism between formal manifolds.  Consequently, the ringed space 
    $(gV', \CO_{G/H}|_{gV'})$ over $\Spec(\BC)$ is a formal manifold  which is isomorphic to $W'\times S_E$, where $V':=(\overline{\pi}\circ \exp)(W')$ is an open subset of $G/H$. Moreover,   together with the morphism
    \be \label{eq:piongU}\pi|_{gU'}:\ (gU',\CO_{G}|_{gU'})\rightarrow (gV',\CO_{G/H}|_{gV'}),\ee
    the formal manifold  $(gV', \CO_{G/H}|_{gV'})$ is a   quotient of  $(gU',\CO_{G}|_{gU'})$ by $H$ (with respect to the right formal action of $H$ on $(gU',\CO_{G}|_{gU'})$).
\end{lemd}	


Since $H$ is closed in $G$,  the underlying topological space of $G/H$ is Hausdorff and paracompact (see \cite[Chapter III, Propositions 2.5.13 and 4.6.13]{Bo}
).

\begin{lemd}\label{lem:GHmani}
The ringed space $G/H$ over $\Spec (\BC)$ is a formal manifold.    
\end{lemd}
\begin{proof}
    Since $G/H$ is Hausdorff and paracompact, the lemma is implied by Lemma  \ref{lem:chartofquo}.
\end{proof}

\begin{lemd}\label{lem:O(G)H}
For every open subset $V$ of $G/H$, the homomorphism \[ \pi^*_V: \  \CO_{G/H}(V) \rightarrow \CO_G(\overline{\pi}^{-1}(V))
\] is a closed embedding.  
\end{lemd}
\begin{proof} 
The lemma follows from  \cite[Lemma 4.6]{CSW1}, Lemmas \ref{lem:quoofsimple} and \ref{lem:chartofquo}. 
\end{proof}

\noindent\textbf{Proof of Theorem \ref{thm:quo}:}
The assertion (b) follows from Lemma \ref{lem:O(G)H}. Using Lemma \ref{lem:chartofquo}, it is clear that the quotient morphism $\pi: G\rightarrow G/H$ is a regular submersion at each point of $G$. Then the assertion (c) follows from \cite[Proposition 3.22]{CSW4}. By Lemma \ref{lem:GHmani},  it remains to prove that $(G/H,\pi)$ is the quotient of $G$ by $H$.

Let $\varphi=(\overline{\varphi},\varphi^*): G\rightarrow M$ be an $H$-invariant morphism. 
Then the  homomorphism $\varphi^*:\CO_M(M)\rightarrow \CO_G(G)$ induces a  unique  continuous homomorphism $ (\varphi')^{*}:\ \CO_M(M)\rightarrow \CO_G(G)^H $ between formal algebras such that the diagram 
     \[\xymatrix{\CO_M(M)\ar[r]^{\varphi^*}\ar[rd]_{(\varphi')^*}&\CO_G(G)\\&\CO_G(G)^H\ar[u]_{\pi^*}}\] commutes. By Lemma \ref{lem:O(G)H} 
    and  \cite[Theorem 5.11]{CSW1}, there is a unique  morphism $\varphi': G/H\rightarrow M$ between formal manifolds determined by $(\varphi')^*$ such that $\varphi'\circ \pi=\varphi$.  
Then  $(G/H, \pi)$ is the quotient of $G$ by $H$, as required. 
\qed


\subsection{Quotients by normal formal Lie subgroups}
 Let  $G$ be a formal Lie group, and let $H$ be a closed regular normal formal Lie subgroup of $G$.
 
 Using Lemma \ref{prop:asfunctors}, 
 an $H$-invariant morphism with respect to the right formal action of $H$ on $G$ given by right multiplication is also $H$-invariant with respect to the left formal action of $H$ on $G$ given by left multiplication.
 Furthermore, the right quotient $(G/H,\pi)$ coincides with the left quotient $(H\backslash G,\pi)$ of $G$ by $H$. 
 
It follows from Lemma \ref{prop:asfunctors} that there is a unique morphism 
\be \label{eq:mofGH} m:\ G/H\times G/H\rightarrow G/H\ee such that the diagram 
\be \label{eq:diaofm0}\begin{CD}
    G\times G @>m>> G \\ @V\pi \times \pi VV @VV\pi V\\ G/H\times G/H@>m>> G/H
\end{CD}\ee
commutes. By abuse of notation, we also write $\varepsilon$ for the composite morphism
\be \label{eq:unitofGH}
\Spec(\BC)\xrightarrow{\varepsilon} G \xrightarrow{\pi} G/H.
\ee

\begin{prpd}\label{prop:quoofnormal}  Let  $G$ be a formal Lie group, and let $H$ be a closed regular normal formal Lie subgroup of $G$.

\noindent(a) The triple $(G/H,m,\varepsilon)$ forms a formal Lie group, and the quotient  morphism $\pi: G\rightarrow G/H$ is a homomorphism of formal Lie groups.

\noindent(b) The 
Lie algebra homomorphism \[d\pi_e: \g\rightarrow \Lie_{\BC}(G/H)\] is surjective with  $\ker(d\pi_e)=\h$. In particular, $\Lie_{\BC}(G/H)=\g/\h$ as Lie algebras. 
\end{prpd}
\begin{proof}
    Using \eqref{eq:diaofm0} and Lemmas \ref{prop:asfunctors},  \ref{lem:normalset}, the assertion (a) is clear. 
    The assertion (b) follows from Lemma 
    \ref{lem:chartofquo}.
    \end{proof}

    
The following corollary is clear by using  Lemma \ref{prop:asfunctors} and Proposition \ref{prop:quoofnormal}. 
\begin{cord}\label{cor:quotoGin}  Let  $G$ be a formal Lie group, and let $H$ be a closed regular normal formal Lie subgroup of $G$. Let  $\varphi:G\rightarrow G'$ be an $H$-invariant homomorphism of formal Lie groups. Then the unique morphism $\varphi': G/H\rightarrow G'$ satisfying $\varphi=\varphi'\circ \pi $ is a homomorphism of formal Lie groups.
\end{cord}

\section{Formal Lie groups associated to Lie pairs}

\subsection{Formal Lie groups associated to Lie pairs} Let $(\q,L,\mathrm{Ad},\iota)$ be a Lie pair. 
Recall from Examples \ref{exam:Lq} and \ref{exam:LieofLq} that there is an infinitesimal formal Lie group  $\mathcal{G}_{\q}$  with $\Lie_{\BC}(\mathcal{G}_{\q})=\q$ as Lie algebras and $\CO_{\mathcal{G}_{\q}}(\{e\})=(\RU(\q))'$ as Hopf formal algebras. 
\begin{prpd}\label{prop:Lltimesq}
   There exists a unique formal action \be\label{eq:LactsonLq}\mu:\  L\curvearrowright \mathcal{G}_{\q}\ee as automorphisms 
     such that the smooth representation $L\curvearrowright \q$ induced by \eqref{eq:LactsonLq}  (as in  Proposition \ref{prop:GonTH}) coincides with  $\Ad: L\curvearrowright \q$.
     
\end{prpd}
\begin{proof}


The representation $\Ad: L\curvearrowright \q$ induces a  representation $ L\curvearrowright \RU(\q)$ which further induces a representation 
\[
L\curvearrowright (\RU(\q))', \quad (g,f)\mapsto g.f.
\]
It is easily checked that the orbit map
\[
\begin{array}{rcl}
  \CO_{\mathcal{G}_{\q}}(\{e\})= (\RU(\q))'  &\rightarrow & \RC^{\infty}(L;(\RU(\q))')= \CO_{L}(L)\wh\otimes_{\pi} (\RU(\q))'=\CO_{L\times \mathcal{G}_{\q}}(L\times \{e\})\\
   f&\mapsto&(g\mapsto g.f)
\end{array}
\]
is a well-defined continuous $\BC$-algebra homomorphism. 
In view of \cite[Theorem 5.11]{CSW1}, the above orbit map corresponds to a morphism  
$\mu: L\times \mathcal{G}_{\q}\rightarrow \mathcal{G}_{\q} $ of formal manifolds.  It is routine to check that $\mu$  is a formal action as automorphisms, and that the action $L\curvearrowright \q$ induced by $\mu$ coincides with $\Ad: L\curvearrowright \q$. This proves the existence of \eqref{eq:LactsonLq},
while the uniqueness is obvious.

\end{proof}

Let $\mu$ be as in Proposition  \ref{prop:Lltimesq}, and we denote by 
\[L\ltimes \mathcal{G}_{\q}\]
 the semidirect product of $L$ and $\mathcal{G}_{\q}$ with respect to $\mu$.
By applying  Proposition \ref{prop:Liealgofsemi} and Lemma \ref{prop:[g,g]},   we have that \be\label{eq:LieofLLq} \Lie_{\BC}(L\ltimes \mathcal{G}_{\q})=\l\ltimes_{\ad}\q\ee as Lie algebras.


\begin{prpd}\label{prop:underlineofLLq}
There are identifications \[\underline {L\ltimes \mathcal{G}_{\q}}=L\times \Spec(\BC)=L\] of  Lie groups. Furthermore,  the morphism \be\label{eq:iotaLto} L=L\times \Spec(\BC)  \xrightarrow{\mathrm{id}_L\times \varepsilon}L\ltimes \mathcal{G}_{\q}\ee  coincides with the morphism $L=\underline {L{\ltimes} \mathcal{G}_{\q}}\xrightarrow{\iota} L\ltimes \mathcal{G}_{\q}$.
\end{prpd}
\begin{proof}
The first assertion follows from  \eqref{eq:underlineofsemi}, while the second assertion is straightforward by using \eqref{eq:defmo}.  
\end{proof}

Note that  $\mathcal{G}_{\l}$ is a formal Lie subgroup of $\mathcal{G}_{\q}$, and $\mathcal{G}_{\l}=L_{\{e\}}$ is also  a normal formal Lie subgroup of $L$  (see Example \ref{exam:Genormal} and Remark \ref{rmk:Ge=Lg}). 
\begin{prpd}\label{prop:LLq/Llasgroup}
\noindent (a) The composition 
\be\label{eq:idtimesi}  \vartheta: \ \mathcal{G}_{\l}\xrightarrow{\mathrm{id}_{\mathcal{G}_{\l}}\times i}\mathcal{G}_{\l} \times \mathcal{G}_{\l}\xrightarrow{\iota\times \iota} L\ltimes \mathcal{G}_{\q} \ee 
is a homomorphism of formal Lie groups which is   an injective immersion as a morphism of  formal manifolds. Furthermore,
the formal Lie subgroup $H_{\l}$ of  $L\ltimes \mathcal{G}_{\q}$ determined by $(\mathcal{G}_{\l},\vartheta)$ (as in \eqref{eq:StoG}) is closed, regular and normal. 
 
\noindent (b) 
The  quotient $((L\ltimes \mathcal{G}_{\q})/\mathcal{G}_{\l},\pi)$ of $L\ltimes \mathcal{G}_{\q}$ by $\mathcal{G}_{\l}$ with respect to the right formal action \[(L\ltimes \mathcal{G}_{\q})\times \mathcal{G}_{\l} \xrightarrow{\mathrm{id}_{L\ltimes \mathcal{G}_{\q}}\times \vartheta} (L\ltimes \mathcal{G}_{\q}) \times (L\ltimes \mathcal{G}_{\q}) \xrightarrow {m} L\ltimes \mathcal{G}_{\q}\]exists. Together with the multiplication morphism as in \eqref{eq:mofGH} and the unit morphism as in \eqref{eq:unitofGH}, $(L\ltimes \mathcal{G}_{\q})/\mathcal{G}_{\l}$ becomes a formal Lie group. Moreover,  
$\Lie_{\BC}((L\ltimes \mathcal{G}_{\q})/\mathcal{G}_{\l})=\q$ as Lie algebras, the quotient morphism 
$\pi: L\ltimes \mathcal{G}_{\q}\to (L\ltimes \mathcal{G}_{\q})/\mathcal{G}_{\l}$ is a homomorphism of formal Lie groups, 
and its differential
\[
d\pi_e:\ \l\ltimes_{\ad}\q \longrightarrow \q
\]
is given by
\be\label{eq:dpil} \tau\otimes\delta_e+\delta_e\otimes\eta\mapsto \iota(\tau)+ \eta \quad\text{($\tau\in \l$ and $\eta\in \q$)}.\ee



 \noindent (c) There is an  identification $\underline{(L\ltimes \mathcal{G}_{\q})/\mathcal{G}_{\l}}=L $ of Lie groups. Moreover,
 the composition \be \label{eq:Ltol} L=L\times \Spec(\BC)\xrightarrow{\mathrm{id}_L\times \varepsilon} L\ltimes \mathcal{G}_{\q}\xrightarrow{\pi} (L\ltimes \mathcal{G}_{\q})/\mathcal{G}_{\l} \ee coincides with the morphism $L=\underline{(L\ltimes \mathcal{G}_{\q})/\mathcal{G}_{\l}}\xrightarrow{\iota} (L\ltimes \mathcal{G}_{\q})/\mathcal{G}_{\l}$.  

\noindent (d)  The morphism
\be\label{eq:underpi=id}  L= \underline{L\ltimes \mathcal{G}_{\q}}\xlongrightarrow{\underline \pi} \underline{(L\ltimes \mathcal{G}_{\q})/\mathcal{G}_{\l}}=L\ee coincides with the identity morphism $\mathrm{id}_{L}$.  
\end{prpd}
In the remainder of this subsection, we prove Proposition \ref{prop:LLq/Llasgroup}. Keep the notation of Proposition \ref{prop:LLq/Llasgroup}. Recall from \eqref{eq:Phi} that $\Phi$ indicates the action by inner automorphisms. 

\begin{lemd}\label{lem:LlonLq}
    The two formal actions \[\mu|_{\mathcal{G}_{\l}\times \mathcal{G}_{\q}}:\  \mathcal{G}_{\l} \times \mathcal{G}_{\q}\xrightarrow{\iota \times \mathrm{id}_{\mathcal{G}_{\q}}} L\times \mathcal{G}_{\q}\xrightarrow{\mu} \mathcal{G}_{\q} \qaq \Phi|_{\mathcal{G}_{\l}\times \mathcal{G}_{\q}}: 
    \ \mathcal{G}_{\l} \times \mathcal{G}_{\q} \xrightarrow{\iota \times \mathrm{id}_{\mathcal{G}_{\q}}} \mathcal{G}_{\q}\times \mathcal{G}_{\q}\xrightarrow{\Phi} \mathcal{G}_{\q}\] of $\mathcal{G}_{\l}$ on $\mathcal{G}_{\q}$ coincide.
\end{lemd}
\begin{proof} 
By using Propositions \ref{prop:Lltimesq}, \ref{prop:GonTH}, \eqref{eq:LieGactsonh} and Lemma \ref{prop:[g,g]},  we have that 
\[\Lie_{\BC}(\mathcal{G}_{\l}\ltimes_{\mu|_{\mathcal{G}_{\l}\times \mathcal{G}_{\q}}} \mathcal{G}_{\q})=\l\ltimes_{\ad}\q=\Lie_{\BC}(\mathcal{G}_{\l}\ltimes_{\Phi|_{\mathcal{G}_{\l}\times \mathcal{G}_{\q}}} \mathcal{G}_{\q}).\] 
It is clear that the differential of the morphism \be\label{eq:idtimesid} \mathrm{id}_{\mathcal{G}_{\l}}\times \mathrm{id}_{\mathcal{G}_{\q}}:\ \mathcal{G}_{\l}\ltimes_{\mu|_{\mathcal{G}_{\l}\times \mathcal{G}_{\q}}} \mathcal{G}_{\q}\rightarrow \mathcal{G}_{\l}\ltimes_{\Phi|_{\mathcal{G}_{\l}\times \mathcal{G}_{\q}}} \mathcal{G}_{\q} \ee  at $(e,e)$ coincides with $\mathrm{id}_{\l\ltimes_{\ad}\q}$. Hence \eqref{eq:idtimesid} is an isomorphism of infinitesimal formal Lie groups. Then by using Lemma \ref{prop:asfunctors}, the lemma follows. 
\end{proof} 

The following lemma is immediate by Lemmas \ref{prop:asfunctors}, \ref{lem:LlonLq} and \ref{prop:semiLiegroups}.
\begin{lemd}\label{lem:LlonLqhom}
     The morphism $\vartheta$ is a homomorphism between formal Lie groups. 
\end{lemd} 

The following lemma follows immediately from Lemmas \ref{lem:inverse} and \ref{lem:difftimes}.
\begin{lemd}\label{lem:dvartheta} 
The differential $d\vartheta_e$ of $ \vartheta: \mathcal{G}_{\l}\rightarrow L\ltimes \mathcal{G}_{\q}$ is given by \[ d\vartheta_e: \ \l\rightarrow \Lie_{\BC}(L{\ltimes} \mathcal{G}_{\q})=\l\ltimes_{\ad}\q, \quad \tau \mapsto \tau\otimes\delta_e-\delta_e\otimes\iota(\tau).\] 
Moreover, the homomorphism $\vartheta$ is an injective immersion of formal manifolds. 
\end{lemd}
In view of Lemmas \ref{lem:LlonLqhom} and \ref{lem:dvartheta}, the  pair $(\mathcal{G}_{\l},\vartheta)$ determines a formal Lie subgroup $H_{\l}$ of $L\ltimes \mathcal{G}_{\q}$ (as in \eqref{eq:StoG}). 
\begin{lemd}\label{lem:Llregu}
The formal Lie subgroup  $H_{\l}$ of $L\ltimes \mathcal{G}_{\q}$ is closed and regular. 
\end{lemd}
\begin{proof} 
It is clear that  $H_{\l}$ is a closed formal Lie subgroup of  $L\ltimes \mathcal{G}_{\q}$. It suffices to prove the map \[d\vartheta^{\mathrm{f}}_e:\ \l/\Lie_{\BC}(\underline{\mathcal{G}_{\l}})\rightarrow \Lie_{\BC}(L{\ltimes} \mathcal{G}_{\q})/ \Lie_{\BC}(\underline {L{\ltimes} \mathcal{G}_{\q}}) \] is injective (\cf \eqref{eq:diotaf} or see  \cite[p.~12]{CSW4}). 
Note that $\Lie_{\BC}(\underline{\mathcal{G}_{\l}})=\{0\}$.
Proposition \ref{prop:underlineofLLq} implies that $\Lie_{\BC}(\underline {L{\ltimes} \mathcal{G}_{\q}})=\l\otimes\delta_e$. 
Then by Lemma \ref{lem:dvartheta},   the map $d\vartheta_e^{\mathrm f}$ coincides with  the injective linear map
\[\l \rightarrow \q, \quad\tau\mapsto -\iota(\tau). \] This finishes the proof.
\end{proof}

Recall that there is a unique formal action $\Phi|_{L\times \mathcal{G}_{\l}}:L\curvearrowright \mathcal{G}_{\l}$
such that the inclusion morphism $\iota: \mathcal{G}_{\l}=L_{\{e\}}\rightarrow L$ is $L$-equivariant, where $L$ carries the formal action of $L$ given by inner automorphisms. 
\begin{lemd}\label{lem:LequiLl} 
    The inclusion  morphism $\iota: \mathcal{G}_{\l}\rightarrow \mathcal{G}_{\q}$ is $L$-equivariant, where $\mathcal{G}_{\l}$ carries the formal action $\Phi|_{L\times \mathcal{G}_{\l}}:L\curvearrowright \mathcal{G}_{\l}$, and $\mathcal{G}_{\q}$ carries the formal action $\mu: L\curvearrowright \mathcal{G}_{\q}$ (see \eqref{eq:LactsonLq}).
\end{lemd}
\begin{proof} 
 By Proposition \ref{prop:Lltimesq} and \eqref{eq:AdofLiegroup}, $\Phi|_{L\times \mathcal{G}_{\l}}$ is the unique formal action of $L$ on $\mathcal{G}_{\l}$  such that the induced smooth representation $L\curvearrowright \l$ is the usual adjoint representation.  
Then the lemma follows from the fact that $\iota: \l \rightarrow \q$ is $L$-equivariant.
\end{proof}




The following lemma follows immediately from Lemmas  \ref{prop:asfunctors}, \ref{lem:LequiLl} and \ref{lem:LlonLq}.
\begin{lemd}\label{lem:Llnor}
 The formal Lie subgroup  $H_{\l}$ of $L\ltimes \mathcal{G}_{\q}$ is normal. 
\end{lemd}

\noindent\textbf{Proof of Proposition \ref{prop:LLq/Llasgroup}:} 
The assertion (a) follows from Lemmas \ref{lem:LlonLqhom}, \ref{lem:dvartheta}, \ref{lem:Llregu} and \ref{lem:Llnor}. 
The assertion (b) then follows from Proposition \ref{prop:quoofnormal} and Lemma \ref{lem:dvartheta}. The assertions (c) and (d) are easy to verify by using Proposition \ref{prop:underlineofLLq}, \eqref{eq:underlinecommdiag} and \eqref{eq:underlineG/H}.
This finishes the proof.\qed

\subsection{Formal Lie groups as formal manifolds}
Let $G$ be a formal Lie group. 
Recall from  \eqref{eq:degaM} the  degree $\deg_e G$ of $G$ at $e$. 
Then we obtain a  formal manifold $\underline{G}^{(\deg_e G)}$ as in \eqref{eq:N(k)}.

\begin{prpd}\label{prop:GtounderGdegG} 
    As formal manifolds, $G$ is isomorphic to 
    $\underline{G}^{(\deg_e G)}$. 
\end{prpd}
\begin{proof}
     Let $E$ be a $\BC$-linear subspace of $\g$ such that \be \label{summand}\g=E \oplus \underline\g.\ee 
Recall the formal submanifold $S_E$ of $G$ from Lemma  \ref{prop:subLs}. Then 
the morphism 	\be \label{eq:thetaonLs}\theta :\ \underline G \times S_{E}\xrightarrow {\iota\times\iota} G\times G \xrightarrow{m} G\ee is $\underline G $-equivariant, where $\underline G\times S_E$ carries the formal action of $\underline G$ as in \eqref{eq:GactsonGtimesM}, and $G$ carries the formal action of $\underline G$ 
given by left multiplication. 
It suffices to prove that the morphism \eqref{eq:thetaonLs}
	is an isomorphism of formal manifolds. 
    
    By Lemmas \ref{lem:inverse} and \ref{prop:subLs},  the differential $d\theta_{(e,e)} $ is bijective. 
Since $\underline{\theta} = \mathrm{id}_{\underline{G}}$, it follows that $d\underline{\theta}_{(e,e)}$ is also bijective.
By using the inverse function theorem (see \cite[Theorem 1.3]{CSW4}) and the $\underline G$-equivariance of $\theta$, it is straightforward to see that \eqref{eq:thetaonLs} is an isomorphism of formal manifolds, as desired. 
\end{proof}

\section{Proof of Theorem \ref{thm:eqFP}}\label{sec:proofofmain}
In this section, we prove Theorem \ref{thm:eqFP} by Proposition \ref{prop:equitoHopf} (identifying formal Lie groups with Hopf formal algebras and with Hopf formal coalgebras) and Proposition \ref{prop:pairtogroup} (identifying formal Lie groups with Lie pairs).


\subsection{Hopf formal algebras and Hopf formal coalgebras}

Recall from Section \ref{sec:tophopf} that for each formal manifold $M$, the LCS $\CO_M(M)$ is a formal algebra, and the LCS $\RD^{-\infty}_c(M;\CO_M)$ is a formal coalgebra. 

For each morphism \[\varphi=(\underline \varphi, \varphi^*):\ (M_1,\CO_{M_1})\rightarrow (M_2,\CO_{M_2})\]
of formal manifolds, the homomorphism $\varphi^*: \CO_{M_2}(M_2)\rightarrow \CO_{M_1}(M_1)$ is a continuous homomorphism between formal algebras (see \cite[Theorem 1.8]{CSW1}), and the transpose \be \label{eq:tvarphi*}{}^t\varphi^*: \ \RD_c^{-\infty}(M_1;\CO_{M_1})\rightarrow \RD_c^{-\infty}(M_2;\CO_{M_2})\ee of $\varphi^*$ is a continuous homomorphism between formal coalgebras. It is obvious that 
\[(\varphi_1\circ \varphi)^*=\varphi^*\circ\varphi_1^* \qaq {}^t(\varphi_1\circ \varphi)^*={}^t\varphi_1^*\circ{}^t\varphi^*\]
where $\varphi_1: M_2\rightarrow M_3$ is another morphism of formal manifolds.

Let $G$ be a formal Lie group. Then we have the continuous homomorphisms 
 \be\label{eq:formalhopfalg}
\begin{aligned}
		&m^*:\CO_G(G)\rightarrow \CO_{G\times G}(G\times G)=\CO_G(G) \wh{\otimes}_{\pi}\CO_G(G)\quad \text{(see \eqref{eq:OM1M2})},\\ 
 	&\varepsilon ^*: \CO_G(G)\rightarrow \BC, \qaq   i^*: \CO_G(G)\rightarrow \CO_G(G)
\end{aligned}
\ee
between formal algebras, 
and the continuous homomorphisms 
\[{}^tm^*: \RD^{-\infty}_c(G;\CO_G)\wh\otimes_{\mathrm i}\RD^{-\infty}_c(G;\CO_G)\rightarrow \RD^{-\infty}_c(G;\CO_G) \quad \text{(see \eqref{eq:DMM=DMDM})},\]
   \[ {}^t\varepsilon^*: \BC\rightarrow \RD^{-\infty}_c(G;\CO_G), \qaq {}^ti^*: \RD^{-\infty}_c(G;\CO_G)\rightarrow \RD^{-\infty}_c(G;\CO_G)
\]
between formal coalgebras.  
\begin{prpd}\label{prop:formalliegrouptohopfalg} Let $G$ be a formal Lie group.

\noindent(a) Together with the comultiplication $m^*$, the counit $\varepsilon^*$ and 
the antipode $i^*$, the formal algebra $\CO_G(G)$ becomes a Hopf formal algebra. 

\noindent (b) Together with the multiplication ${}^tm^*$, the unit ${}^t\varepsilon^*$ and 
the antipode ${}^ti^*$, the formal coalgebra  $\RD^{-\infty}_c(G;\CO_G)$ becomes a Hopf formal coalgebra. 
\end{prpd}
\begin{proof}
The assertions follow from the commutative diagrams in Definition \ref{df:Lirgroup}.
\end{proof}


\begin{prpd}\label{prop:equitoHopf}
\noindent(a) The assignment \be\label{eq:GtoO(G)}  G\mapsto \CO_G(G), \quad (\varphi: G_1\rightarrow G_2)\mapsto (\varphi^*: \CO_{G_2}(G_2)\rightarrow \CO_{G_1}(G_1))\ee is a functor from the category of formal Lie groups to the opposite category of the category of Hopf formal algebras. Furthermore, the functor is an equivalence between categories. 

 \noindent(b)  The assignment \be \label{eq:GtoD(G)} G\mapsto \mathrm{D}^{-\infty}_c(G;\CO_G), \quad (\varphi: G_1\rightarrow G_2)\mapsto ({}^t\varphi^*: \mathrm{D}^{-\infty}_c(G_1;\CO_{G_1})\rightarrow \mathrm{D}^{-\infty}_c(G_2;\CO_{G_2}))\ee is a functor from the category of formal Lie groups to the category of Hopf formal coalgebras. Furthermore, the functor is an equivalence between categories.
\end{prpd}
\begin{proof}   
   The first assertions in (a) and (b) are obvious. It follows from \cite[Theorem 5.11]{CSW1} that the functor \eqref{eq:GtoO(G)} is an equivalence between categories. 
Note that all formal algebras and formal  coalgebras are reflexive as LCS (see \cite[Proposition 4.8]{CSW1}). Thus, by taking strong duals, there is an equivalence between 
    the opposite category of the category of formal algebras and the category of formal coalgebras.  This implies that the functor \eqref{eq:GtoD(G)} is also an equivalence of categories.    
\end{proof}
\subsection{Lie pairs}
Let $(d\varphi,\underline \varphi):   (\q_1,L_1)\rightarrow (\q_2,L_2) $ be a homomorphism of Lie pairs (see Definition \ref{df:morLiepair}). By  the formal Lie theory theorem (see \cite[(14.2.3)]{Ha} and \cite[Theorem 5.18]{MM} or Theorem \ref{thm:eqFPpre}), 
there is a unique homomorphism \[ \wh{d\varphi}:\ \mathcal{G}_{\q_1}\rightarrow \mathcal{G}_{\q_2}\] between infinitesimal formal Lie groups such that the differential of $\wh{d\varphi}$ at $e$ is $d\varphi$.
By using Lemma \ref{prop:asfunctors}, one verifies that 
the morphism \[\underline \varphi\times \wh{d\varphi}: \ L_1\ltimes \mathcal{G}_{\q_1}\rightarrow L_{2}\ltimes \mathcal{G}_{\q_2}\]  
is a homomorphism between formal Lie groups.  Recall from Proposition \ref{prop:LLq/Llasgroup} that $ (L_{1}\ltimes \mathcal{G}_{\q_1})/ \mathcal{G}_{\l_1}$ and $(L_{2}\ltimes \mathcal{G}_{\q_2})/ \mathcal{G}_{\l_2}$ are formal Lie groups. The following result is clear by using Lemma \ref{prop:asfunctors}.
\begin{lemd}
\label{prpd:L/Lto}
  There is a unique morphism \be \label{eq:L/Lto}\varphi:\ (L_1\ltimes \mathcal{G}_{\q_1})/\mathcal{G}_{\l_1} \rightarrow (L_2\ltimes \mathcal{G}_{\q_2})/\mathcal{G}_{\l_2} \ee 
  of formal manifolds such that the diagram \[\begin{CD}
      L_{1}\ltimes \mathcal{G}_{\q_1} @>\underline \varphi \times \wh{d\varphi} >>L_{2}\ltimes \mathcal{G}_{\q_2}\\ @V\pi VV @VV\pi V \\ (L_{1}\ltimes \mathcal{G}_{\q_1})/ \mathcal{G}_{\l_1} @>\varphi>> (L_{2}\ltimes \mathcal{G}_{\q_2})/ \mathcal{G}_{\l_2}
  \end{CD}\] commutes.  Furthermore, the morphism \eqref{eq:L/Lto} is a homomorphism of formal Lie groups. 
\end{lemd}

Recall from Proposition \ref{prop:GtoLiepair} that the assignment
\be\label{eq:functorGtopair} G\mapsto (\g,\underline G), \qaq (\varphi:G_1\rightarrow G_2)\mapsto((d\varphi_e,\underline \varphi ): (\g_1,\underline {G_1})\rightarrow (\g_2,\underline{G_2}))\ee is a functor from the category of formal Lie groups to that of Lie pairs.

\begin{prpd}\label{prop:pairtogroup}
\noindent(a) The assignment
\be\label{eq:funcpairtogroup}
\begin{aligned}
		(\q,L)&\mapsto (L\ltimes \mathcal{G}_{\q})/\mathcal{G}_{\l}, \\ 
 	((d\varphi,\underline \varphi): (\q_1,L_1)\rightarrow (\q_2,L_2))&\mapsto (\varphi: (L_1\ltimes \mathcal{G}_{\q_1})/\mathcal{G}_{\l_1}\rightarrow (L_2\ltimes \mathcal{G}_{\q_2})/\mathcal{G}_{\l_2})
\end{aligned}
\ee
forms a functor from the category of Lie pairs to that of formal Lie groups, where $\varphi$ is as in Lemma \ref{prpd:L/Lto}.      

\noindent  (b) The functors \eqref{eq:functorGtopair} and  \eqref{eq:funcpairtogroup} are quasi-inverse of each other.
\end{prpd}
In the remainder of this subsection, we prove Proposition \ref{prop:pairtogroup}. The assertion (a) is straightforward, and it remains to prove the assertion (b). 

 
By Propositions \ref{prop:LLq/Llasgroup} and \ref{prop:GtoLiepair}, for every Lie pair $(\q, L)$, the pair \[(\Lie_{\BC}((L\ltimes \mathcal{G}_{\q})/\mathcal{G}_{\l}),\underline{(L\ltimes \mathcal{G}_{\q})/\mathcal{G}_{\l}})\] is a Lie pair with
$\Lie_{\BC}((L\ltimes \mathcal{G}_{\q})/\mathcal{G}_{\l})=\q $ as Lie algebras and $\underline{(L\ltimes \mathcal{G}_{\q})/\mathcal{G}_{\l}}=L $ as Lie groups. 
\begin{lemd}\label{lem:qLidento}
    There is an identification \[(\q, L) = (\Lie_{\BC}((L\ltimes \mathcal{G}_{\q})/\mathcal{G}_{\l}), \underline{(L\ltimes \mathcal{G}_{\q})/\mathcal{G}_{\l}})\] of Lie pairs.
\end{lemd}
\begin{proof} In view of Proposition \ref{prop:GtoLiepair},
it suffices to verify that the differential  
\be \label{eq:diffunderLltimesquo} \l=
\Lie_{\BC}(\underline{(L\ltimes \mathcal{G}_{\q})/\mathcal{G}_{\l}})\longrightarrow \Lie_{\BC}((L\ltimes \mathcal{G}_{\q})/\mathcal{G}_{\l})=\q
\ee of the inclusion morphism $\underline{(L\ltimes \mathcal{G}_{\q})/\mathcal{G}_{\l}}\rightarrow {(L\ltimes \mathcal{G}_{\q})/\mathcal{G}_{\l}}$
coincides with $\iota:\l\to\q$, and that the adjoint representation
\be \label{eq:Adcomplex}
L= \underline{(L\ltimes \mathcal{G}_{\q})/\mathcal{G}_{\l}}\curvearrowright \Lie_{\BC}((L\ltimes \mathcal{G}_{\q})/\mathcal{G}_{\l})=\q
\ee as in \eqref{eq:Addef}
coincides with $\Ad:L\curvearrowright\q$.
For the first assertion, by Proposition \ref{prop:LLq/Llasgroup} (b) and (c),  
the map \eqref{eq:diffunderLltimesquo} 
 coincides with the differential \[\l\rightarrow \l\ltimes_{\ad}\q\rightarrow \q,\quad \tau\mapsto \tau\otimes\delta_e\mapsto\iota(\tau)\] of \eqref{eq:Ltol}.

For the second one, 
note that $\underline{L\ltimes \mathcal{G}_{\q}}=L$ (see Proposition \ref{prop:underlineofLLq}). Then, using Proposition \ref{prop:underlineofLLq},  Proposition \ref{prop:LLq/Llasgroup} (c) and Lemma \ref{prop:asfunctors}, the homomorphism \be \label{eq:Lto/L}\mathcal{G}_{\q}=\Spec(\BC)\times \mathcal{G}_{\q}\xrightarrow{\varepsilon \times \mathrm{id}_{\mathcal{G}_{\q}}} L\ltimes \mathcal{G}_{\q}\rightarrow (L\ltimes \mathcal{G}_{\q})/\mathcal{G}_{\l} \ee
is $L$-equivariant, where $L$ acts on $\mathcal{G}_{\q}$ formally by \eqref{eq:LactsonLq},  $L\ltimes \mathcal{G}_{\q}$ and $(L\ltimes \mathcal{G}_{\q})/\mathcal{G}_{\l}$
carry the formal action of $L$ by \[L\times (L\ltimes \mathcal{G}_{\q})\xrightarrow{\eqref{eq:iotaLto}\times \mathrm{id}_{L\ltimes \mathcal{G}_{\q}}} ({L\ltimes \mathcal{G}_{\q}})\times (L\ltimes \mathcal{G}_{\q})\xrightarrow{\Phi} L\ltimes \mathcal{G}_{\q}\]and \[L\times((L\ltimes \mathcal{G}_{\q})/\mathcal{G}_{\l})\xrightarrow{\eqref{eq:Ltol}\times \mathrm{id}_{(L\ltimes \mathcal{G}_{\q})/\mathcal{G}_{\l}}} ((L\ltimes \mathcal{G}_{\q})/\mathcal{G}_{\l})\times ((L\ltimes \mathcal{G}_{\q})/\mathcal{G}_{\l})\xrightarrow{\Phi} (L\ltimes \mathcal{G}_{\q})/\mathcal{G}_{\l}, \] respectively. 
Therefore, by Proposition \ref{prop:GonT}, the differential  \[\q \rightarrow \l\ltimes_{\ad}\q \rightarrow \q , \quad \eta \mapsto \delta_e\otimes \eta \mapsto \eta\] of \eqref{eq:Lto/L} is $L$-equivariant, where the first $\q$ carries the representation $\Ad: L\curvearrowright \q$ by Proposition \ref{prop:Lltimesq}, and  the second $\q$ carries the representation \eqref{eq:Adcomplex}. This completes the proof.  
\end{proof}

By Proposition 
\ref{prop:LLq/Llasgroup}, for every formal Lie group $G$,
the quotient $((\underline G\ltimes \mathcal{G}_{\g})/\mathcal{G}_{\underline \g},\pi) $
of $\underline G\ltimes \mathcal{G}_{\g}$ by $\mathcal{G}_{\underline \g}$ with respect to the right  formal action 
\be\label{eq:actionofLunderg} (\underline G\ltimes \mathcal{G}_{\g})\times \mathcal{G}_{\underline \g}\xrightarrow{\mathrm{id}_{\underline G\ltimes \mathcal{G}_{\g}}\times \vartheta} (\underline G\ltimes \mathcal{G}_{\g})\times (\underline G\ltimes \mathcal{G}_{\g})\xrightarrow{m} \underline G\ltimes \mathcal{G}_{\g}\ee
exists. In addition, $(\underline G\ltimes \mathcal{G}_{\g})/\mathcal{G}_{\underline \g}$ is a formal Lie group and  the quotient morphism $\pi$ is a homomorphism between formal Lie groups.


Using Lemma \ref{prop:asfunctors}, for every formal Lie group $G$, 
the composition \be  \label{eq:homooftoG}\psi:\ \underline G\ltimes \mathcal{G}_{\g}\xrightarrow{\iota\times\iota} G\times G \xrightarrow m G\ee  is a homomorphism of formal Lie groups.  The following lemma follows immediately from Lemmas \ref{prop:asfunctors} and \ref{lem:LlonLq}.  
\begin{lemd}\label{lem:Gequalto}
 	The homomorphism \eqref{eq:homooftoG} is $\mathcal{G}_{\underline \g}$-invariant (with respect to the right formal action \eqref{eq:actionofLunderg}). 
 \end{lemd}
By Corollary \ref{cor:quotoGin} and Lemma \ref{lem:Gequalto}, the morphism \be\label{eq:1homotoG} \psi'=(\underline{\psi'}, {\psi'}^*):\ (\underline G\ltimes \mathcal{G}_{\g})/\mathcal{G}_{\underline{\g}}\rightarrow G \ee induced by \eqref{eq:homooftoG}  is a homomorphism between formal Lie groups.
\begin{lemd}\label{lem:G}
	The homomorphism \eqref{eq:1homotoG}
    is an isomorphism between formal Lie groups.
\end{lemd}\begin{proof}
Recall from Proposition \ref{prop:LLq/Llasgroup} that 
\[
\underline {(\underline G\ltimes \mathcal{G}_{\g})/\mathcal{G}_{\underline{\g}}}=\underline G\]
as Lie groups, and 
\[
\Lie_{\BC}((\underline G\ltimes \mathcal{G}_{\g})/\mathcal{G}_{\underline{\g}})=\g \]
as Lie algebras.
We claim that  $\underline {\psi'}=\mathrm{id}_{\underline G}$ and  $d\psi'_e=\mathrm{id}_{\g}$. Then, the lemma follows from the inverse function theorem (see \cite[Theorem 1.3]{CSW4}) and the fact that $\psi'$ is $\underline G$-equivariant,
where $(\underline G\ltimes \mathcal{G}_{\g})/\mathcal{G}_{\underline{\g}}$ and $G$ carry the formal actions of $\underline G$ given by left multiplication.
 It remains to prove the claim.


By \eqref{eq:underlinecommdiag} and \eqref{eq:chainruleofdvarphi},  the diagrams 
 \be\label{eq:Gquol2} \xymatrix{ \underline G=\underline{\underline G\ltimes \mathcal{G}_{\g}}\ar[d]_{\underline\pi}\ar[rrd]^{\underline \psi}\\ \underline G=\underline{(\underline G\ltimes \mathcal{G}_{\g})/\mathcal{G}_{\underline \g}} \ar[rr]^{\underline{\psi'}}&& \underline G } \qaq \ \ \xymatrix
	{\underline{\g}\ltimes_{\ad}\g \ar[drr]^{d\psi_{(e,e)}}\ar[d]_{d\pi_{(e,e)}} \\ \g\ar[rr]_{d\psi'_{(e,e)}} &&\g }
 \ee
commute. By Proposition \ref{prop:LLq/Llasgroup} (d) and Proposition  \ref{prop:underlineofLLq}, it is clear that 
$\underline \pi=\mathrm{id}_{\underline G}=\underline \psi$. 
Hence, $\underline{\psi'}=\mathrm{id}_{\underline G}$.
On the other hand, by Lemma \ref{lem:inverse}, $d\psi_{(e,e)}$ is given by \[d\psi_{(e,e)}:\ \underline{\g}\ltimes_{\ad}\g\rightarrow \g,\quad \tau\otimes\delta_e+\delta_e\otimes\eta\mapsto \tau+\eta.\]  This, together with \eqref{eq:dpil}, implies that $d\psi'_{(e,e)}=\mathrm{id}_{\g}$. Then the claim follows. 
\end{proof}

By Lemmas \ref{lem:qLidento} and \ref{lem:G}, the assertion (b) in Proposition \ref{prop:pairtogroup} follows. This finishes the proof of Proposition \ref{prop:pairtogroup}.




\appendix
\section{Topological Hopf Algebras}

In this appendix, we collect some basic definitions of topological Hopf $\mathbb{C}$-algebras that are used throughout the paper.

\subsection{Topological Tensor Products}
Throughout this paper,  by an LCS, we mean a 
locally convex topological vector space over $\C$, which may or may not be Hausdorff. 

For two LCS $E$ and $F$, the algebraic tensor product $E \otimes F$ admits two natural locally convex topologies: the inductive tensor product $E \otimes_{\mathrm{i}} F$ and the projective tensor product $E \otimes_{\pi} F$. Explicit definitions can be found in \cite{Gr}.
We denote by
\[
E\widehat{\otimes}_\Ri F \quad \text{and} \quad E\widehat{\otimes}_\pi F
\]
the completions of  (the maximal Hausdorff quotients of)   $E\otimes_\Ri F$ and $E\otimes_\pi F$, respectively.


Given two continuous linear maps $\phi_1:E_1\to F_1$ and $\phi_2:E_2\to F_2$ between LCS, we denote by $\phi_1\otimes\phi_2$ the continuous linear map on the various topological tensor products (or their completions) obtained by the tensor product of $\phi_1$ and $\phi_2$. For example, we have the maps
\[
\phi_1\otimes\phi_2:\ E_1\otimes_\pi E_2 \longrightarrow F_1\otimes_\pi F_2,\qaq
\phi_1\otimes\phi_2:\ E_1\widehat{\otimes}_\Ri E_2 \longrightarrow F_1\widehat{\otimes}_\Ri F_2.
\]

\subsection{Topological Hopf Algebras}\label{sec:topHopfalg}

In this subsection, we use the subscript $*$ to denote either $\Ri$ or $\pi$ so that  $\widehat{\otimes}_{*}$ denotes either $\widehat{\otimes}_\Ri$ or $\widehat{\otimes}_\pi$.

\begin{dfn}\noindent (a) A \textbf{$\widehat{\otimes}_{*}$-algebra} is a complete LCS $A$ together with:
\begin{itemize}
    \item a continuous associative multiplication $m: A\widehat{\otimes}_{*}A \to A$; and
    \item a unit $\eta: \mathbb{C} \to A$.
\end{itemize}

\noindent (b) A \textbf{$\widehat{\otimes}_{*}$-coalgebra} is a complete LCS $C$ together with:
\begin{itemize}
    \item a continuous coassociative comultiplication $\Delta: C \to C\widehat{\otimes}_{*}C$;
    \item a continuous counit $\varepsilon: C \to \mathbb{C}$.
\end{itemize}
\end{dfn}
Here ``associative multiplication" (resp. \,``coassociative comultiplication'') means that $m$ (resp.\,$\Delta$) is linear and the diagram 
\[
  \begin{CD}
  A\widehat{\otimes}_{*}A\widehat{\otimes}_{*}A &@>m\otimes \mathrm{id}_A >> & A\widehat{\otimes}_{*}A\\
  @V\mathrm{id}_A \otimes m VV && @VV m V\\
  A\widehat{\otimes}_{*}A &@> m >>  &A \end{CD} \quad \bigg(\text{resp.}\quad  \begin{CD}
  C\widehat{\otimes}_{*} C\widehat{\otimes}_{*}C &@ <\Delta\otimes \mathrm{id}_C <<& C\widehat{\otimes}_{*}C\\
  @A\mathrm{id}_C \otimes \Delta AA && @AA\Delta  A \\
 C\widehat{\otimes}_{*}C &@<\Delta << &C
  \end{CD}\quad \bigg)
\]
commutes ($\mathrm{id}$ indicates  the identity map); and ``unit" (resp.\,``counit'') means that  $\eta$ (resp.\,$\varepsilon$) is linear and the diagram

\[
\xymatrix{
\mathbb C\otimes A \ar[r]^{\eta\otimes \mathrm{id}_A} \ar[rd]_{=} & A\widehat{\otimes}_{*}A  \ar[d]^{m} &A \otimes \mathbb C \ar[l]_{\mathrm{id}_A\otimes \eta }\ar[ld]^{=} \\
&A  
}\quad \bigg(\text{resp.}\quad \xymatrix{
\mathbb C\otimes C & C\widehat{\otimes}_{*}C  \ar[l]_{\varepsilon\otimes\mathrm{id}_C }\ar[r]\ar[r]^{\mathrm{id}_C \otimes \varepsilon } & C\otimes \mathbb C \\
&C\ar[ul]^{=}\ar[u]_{\Delta} \ar[ur]_{=}
}\quad
\bigg)
\]
commutes. 
\begin{dfn}
A \textbf{Hopf $\widehat{\otimes}_{*}$-algebra} is a complete LCS $\mathcal{H}$ together with continuous linear maps:
\begin{itemize}
    \item multiplication $m: \mathcal{H}\widehat{\otimes}_{*}\mathcal{H} \to \mathcal{H}$;
    \item unit $\eta: \mathbb{C} \to \mathcal{H}$;
    \item comultiplication $\Delta: \mathcal{H} \to \mathcal{H}\widehat{\otimes}_{*}\mathcal{H}$;
    \item counit $\varepsilon: \mathcal{H} \to \mathbb{C}$;
    \item antipode $S: \mathcal{H} \to \mathcal{H}$,
\end{itemize}
satisfying the usual axioms for a Hopf algebra (\cf \cite[Section 3.1]{CP}).
\end{dfn}

\section*{Acknowledgements}
Fulin Chen was supported by the National Natural Science Foundation of China (Grant Nos. 12131018 and 12471029). Binyong Sun was supported in part by National Key R \& D Program of China (Grant No. 2022YFA1005300) and New Cornerstone
Science Foundation. 

\end{document}